\documentclass[12pt,a4paper]{article}
\usepackage[utf8]{inputenc}
\usepackage[T1]{fontenc}
\usepackage[english]{babel}
\usepackage{amsmath}
\usepackage{amssymb,amstext,hyperref,bbm}
\usepackage{mathrsfs}
\usepackage{amsthm}

\usepackage{multicol}
\usepackage{enumerate}
\usepackage{fancyhdr}

\usepackage{color}
\usepackage{graphicx}
\usepackage{caption}

\usepackage{geometry}

\newtheorem{thm}{Theorem}
\newtheorem{lem}{Lemma}

\newtheorem{prop}{Proposition}

\def\bit{\begin{itemize}}
	\def\eit{\end{itemize}}
\def\bdes{\begin{description}}
	\def\edes{\end{description}}

\def\beq{\begin{equation}}
\def\eeq{\end{equation}}

\def\ben{\begin{enumerate}}
	\def\een{\end{enumerate}}

\def\beqar{\begin{eqnarray}}
\def\eeqar{\end{eqnarray}}

\def\beqarr{\begin{eqnarray*}}
	\def\eeqarr{\end{eqnarray*}}
\def\ZZ{{\mathbb Z}}       %bold Z
\def\RR{{\mathbb R}}  % bold R
\def\NN{{\mathbb N}}

\newcommand{\po}{\left(}
\newcommand{\pf}{\right)}

\newcommand{\R}{\mathbb R}

\newcommand{\dd}{\text{d}}
 
\begin{document}
 	\title{\textbf{Strongly self-interacting processes on the circle}}
 	\author{Carl-Erik Gauthier\footnote{Department of Mathematics, Box 354350, University of Washington, Seattle, WA 98195, United States. Email:
 			carlgaut@uw.edu} and Pierre Monmarch\'e\footnote{Sorbonne Universit\'e, 4 place Jussieu, 75 005 Paris, France. Email:
 			pierre.monmarche@upmc.fr}}%\\Institut de Math\'ematiques\\Universit\'e de Neuch\^atel, Switzerland
 	%	\thanks{We acknowledge financial support from the  Swiss National Science Foundation Grant *** }}
 	\maketitle
 	\date{}

\begin{abstract}
The purpose of this paper is to investigate the long time behaviour for a self-interacting diffusion   and a self-interacting velocity jump process. While the diffusion case has already been studied for some particular potential function, the second one, which belongs to the family of piecewise deterministic processes, is new.
 		     
 	Depending on the underlying potential function's shape, we prove either the almost sure convergence or the recurrence for a natural extended process given by a change a variable.\bigskip

%\begin{keywords}Self-Interacting Markov processes; diffusions; PDMP; ergodicity; almost-sure convergence
%\end{keywords}
%\begin{classcode}60K35; 60J25; 60H10; 60J75; 60J60.\end{classcode}\bigskip
\end{abstract}

\section{Introduction}
%%%%%%%%%%%%%%%%%%%%%%%%%%%%%%%%%%%%%%%%%%%%%%%%%%%%%%%%%%%%%%%%%%%%%%%%%%%%%%%%%%%%%%%%%%%%%%%%%%%%%%%%%%%%%%%%%%%%

Our aim is to study the effect of the addition of a self-interaction mechanism to two initially Markovian dynamics. The first one is the classical Fokker-Planck diffusion $X\in\mathbb R$ that solves the SDE
\begin{eqnarray}\label{EqIntro}
d X_t &=& d B_t - V'(X_t) dt,
\end{eqnarray}
where $(B_t)_{t\geqslant 0}$ is a standard Brownian motion on $\mathbb{R}$. Namely $X$ is the Markov process with generator
\begin{eqnarray*}
L f(x) &=&   \frac12 f''(x) -   V'(x)  f'(x).
\end{eqnarray*}
We recall that the generator of a Markov process $(Z_t)_{t\geq 0}$ is formally defined by
\begin{eqnarray*}
L f(z) &=& (\partial_t)_{| t=0}\ \mathbb E\po f(Z_t)\ |\ Z_0 = z\pf.
\end{eqnarray*}
Equation \eqref{EqIntro} can be seen as the generalization of the classical Ornstein-Uhlenbeck process, obtained for $V'(x)=\lambda x,\, \lambda>0$.

The second one is the velocity jump process $(X,Y)\in\mathbb R\times \{-1,1\}$ which is the piecewise deterministic Markov process (PDMP) introduced in \cite{MonmarchePDMP} with generator
\begin{eqnarray*}
Lf(x,y) &=& y \partial_x f(x,y) + \po \lambda + \po y V'(x) \pf_+\pf \po f(x,-y) - f(x,y)\pf
\end{eqnarray*} 
where $\lambda >0$ is constant and $(\ )_+$ denotes the positive part. A trajectory of the process is defined as follows: starting from an initial state $(x,y)$, the process follows the deterministic flow $(X_t,Y_t)=(x+ty,y)$ up to a random time  $T$ with cumulative distribution $\mathbb P(T>s) = \exp[-\lambda  s + \int_0^s  ( y V'(x+uy))d u]$. At time $T$, the velocity is reversed, i.e. $Y_T = -y$, while the position is continuous, i.e. $X_T = x + T y$. By the Markov property, $(X_T,Y_T)$  can then be taken as a new initial state, from which the process again follows free transport up to a new random jump time, etc., and the full trajectory is defined by induction (see \cite{MonmarchePDMP} and Section \ref{SectionHittingPDMP} for details).

In both cases (diffusion or PDMP), if we suppose that the potential $V$ is sufficiently coercive at infinity, $X$ is ergodic and its law converges to the Gibbs measure with density proportional to $e^{-V}$. Note that when the rate of jump $\lambda$ goes to infinity and time is correctly accelerated, the velocity jump process (more precisely its first coordinate) converges to the Fokker-Planck diffusion (see \cite{FGM}).

In both cases we want to replace the potential $V(X_t)$ by a self-interacting potential
\begin{eqnarray*}
V_t(X_t) &=& \int_0^t W(X_t,X_s) d s
\end{eqnarray*}
where $W$ is a symmetric interaction potential. In other words $V_t(X_t)$ depends both on the current position $X_t$ and the (non-normalized) occupation measure $\int_0^t \delta_{X_s} ds$. This is a strong self-interaction, by contrast with the weak self-interaction such as studied in \cite{BLR} where the self-interacting potential is a function of $X_t$ and of the normalized occupation measure $\frac1t\int_0^t \delta_{X_s} ds$.

 Self-Interacting processes belong to the family of \textit{path-dependent} processes. The particularity of such processes is their lack of Markov property since the past modifies the environment that drives the particle. New phenomena may arise in their long time behavior, which would be impossible without the path-dependency. 

A first example of strong self-interaction is  the linear one, that correspond to $W(x,y) = \frac12(x-y)^2$. M.Cranston and Y.Le Jan proved in 1995 (see \cite{cranston}) that the solution of the SDE
 \begin{equation}\label{eq1INTRO}
 dX_t =dB_t -\left(\int_0^t (X_t - X_s) ds\right)d t
 \end{equation}
almost surely converges to a Gaussian random variable as $t$ goes to infinity.  Later, S.Herrmann and B.Roynette extended this result to a broader class of potentials of the form $W(x,y) = V(x-y)$ with $V$ convex (see \cite{HerrRoy}). In the case of the circle, the first author obtained the same result (almost sure convergene toward a random variable) for the interaction potential $W(x,y)=-\cos(x-y)$ (see \cite{G}). In all these cases the particle is attracted by its past.  

 In \cite{CEGMB}, M.Bena\"im and the first author considered the repulsive case, in which the particle is repelled by its past trajectory. More precisely they studied a self-repelling diffusion on a compact manifold where $W$ can be decomposed as
 \begin{eqnarray*} 
  W(x,y) &=& \sum_{i=1}^n a_i e_i(x)e_i(y) 
 \end{eqnarray*}
with the $a_i$'s being positive numbers and the $e_i$'s being eigenfunctions of the Laplace operator on the manifold. The basic example on the circle would be $W(x,y) = \cos(x-y) = \cos(x)\cos(y)+\sin(x)\sin(y)$. This assumption on the $e_i$'s yields an explicit formula for the invariant measure of the Markov process $\po X_t,\po\int_0^t e_i(X_s) ds\pf_{i=1..n}\pf$. 

The aim of the present work is to investigate the case where the $e_i$'s are not eigenfunctions of the Laplace operator. On the other hand we restrict the study (in dimension 1) to the case $n=1$, namely we take a potential of the form
\[W(x,y) = F(x)F(y)\]
with moreover $F$ smooth and $2\pi$-periodic, so that we consider $x\in \mathbb{S}^1=\RR /2\pi\ZZ$. Following \cite{CEGMB}, we set
\begin{eqnarray}\label{U}
U_t  & = & \int_0^t{F(X_s)}d s,
\end{eqnarray} which reduces the study of the non-Markovian process to the study of some Markov process on an extended space. This restriction should be seen as a first step toward the analysis of the more general situation.

As a consequence, in this paper we study the Markov processes $(X,U)$ on $\mathbb S^1\times\RR$ and $(X,U,Y)$ on $\mathbb S^1\times \RR\times\{-1,1\}$ with respective generators
\begin{eqnarray}\label{MODELEdiff}
L_1 f(x,u) &=& \frac12 \partial_x^2 f(x,u) - u F'(x)\partial_x f(x,u) + F(x) \partial_u f(x,u)\end{eqnarray} 
and 
\begin{eqnarray}\label{MODELPDMP}
\lefteqn{L_2 f(x,u,y) =}\notag \\ && y \partial_x f(x,u,y)+ F(x) \partial_u f(x,u,y) + \po \lambda + \po y u F'(x) \pf_+\pf \po f(x,u,-y) - f(x,u,y)\pf.
\end{eqnarray}
In both cases we call $X$ the position, $U$ the auxiliary variable and, in the case of the velocity jump process, $Y$ the velocity. Remark that \eqref{U} would imply that $U_0$ is always 0, but from  now on we consider the general case of the processes with generators \eqref{MODELEdiff} and \eqref{MODELPDMP} with any initial condition $U_0\in\mathbb R$. The following assumptions are supposed to hold throughout all the paper:
\begin{itemize}
\item The function $F:\mathbb S^1 \rightarrow\RR$  is non-constant, $\mathcal C^\infty$, changes signs, and $F'(x)=0$ implies $F(x)\neq 0$. Moreover for all $x\in\mathbb S^1$ there exists $k\geq 1$ such that $F^{(k)}(x)\neq 0$. In particular the critical points of $F$ are isolated points.
\end{itemize}
The assumption that $F$ has no critical point $x$ with $F(x)=0$ and is nowhere flat is made for simplicity: otherwise, different behaviours may arise and many cases would have to be distinguished. We concentrate here on the generic case. Throughout this paper, we consider the discrete sets
\begin{eqnarray*}
M(F,+)&=&\{x\in \mathbb{S}^1\mid x \text{ is a local maximum of $F$ and } F(x)>0 \}\\
M(F,-)&=&\{x\in \mathbb{S}^1\mid x \text{ is a local maximum of $F$ and } F(x)<0 \}\\
m(F,+)&=&\{x\in \mathbb{S}^1\mid x \text{ is a local minimum of $F$ and } F(x)>0 \} \\
m(F,-)&=& \{x\in \mathbb{S}^1 \mid x \text{ is a local minimum of $F$ and } F(x)<0 \}
\end{eqnarray*}  
and $\mathcal M = M(F,-)\cup m(F,+)$. Recall the total variation distance between two probability laws $\mu$ and $\nu$ is
\begin{eqnarray*}
d_{TV}\po \mu,\nu\pf & =& \inf\left\{ \mathbb P\po \Xi_1\neq \Xi_2\pf \, : \ Law(\Xi_1)=\mu,\ Law(\Xi_2)=\nu\right\}
\end{eqnarray*}
and a measure $\mu$ is said invariant for a Markov process $(Z_t)_{t\geq 0}$ if $\{Law(Z_0) = \mu\}$ implies $\{\forall t\geq 0, Law(Z_t) = \mu\}$.
We say that the law of $(Z_t)_{t\geq 0}$ converges exponentially fast to $\mu$ in the total variation sense if there exist $C,\rho>0$, that may depend on the law of $Z_0$, such that for all $t\geq 0$
\[ d_{TV}\po Law(Z_t),\mu\pf \leq C e^{-\rho t}.\]
Finally, we say that a random variable $Z$ \textit{admits an exponential moment}  if there exists $\theta > 0$ such that \[ \mathbb{E}(e^{\theta \vert Z\vert}) <\infty.\]
Our main result is the following:

\begin{thm}\label{ThmPrincipal}\

\begin{enumerate}
\item If $\mathcal M = \emptyset$, then each of the the processes $(X,U)$ with generator \eqref{MODELEdiff} and $(X,U,Y)$ with generator \eqref{MODELPDMP} admits a unique invariant measure with full support. If the law of $U_0$ admits an exponential moment then the process converges exponentially fast in the total variation distance sense to this invariant measure.
\item  If $\mathcal M \neq \emptyset$, then, in both cases, the position $X_t$ almost surely converges to a point of $\mathcal M$, as $t$ goes to infinity. Any point of $\mathcal M$ has a positive probability to be the limit of $X$.
\end{enumerate}
\end{thm}

Before proceeding to its proof, let us mention why this result may be expected. Suppose that, at some time, $U>0$. Then, as long as $U$ is large enough, the force $U_t F'(X_t)$ tends to confine $X$ close to the minima of $F$. If these minima are all negative, while $X$ stays in their neighbourhood, $U$ decreases, up to some point where it becomes negative. From then the effect of the force is reversed, $X$ is attracted by the maxima of $F$, and the same mechanism comes into play with $U$ and $F$ changed to $-U$ and $-F$. In some sense $X$ and $U$ have then an inhibitory effect one on the other.

 On the other hand if $X$ falls in the neighborhood of a positive minimum of $F$ while $U>0$ (the case of a negative maximum with $U<0$ being symmetric) then, as long as it stays there, $U$ increases, which make it more and more unlikely for $X$ to escape away from the minimum, so that eventually there is a positive probability that $X$ never leaves and  $U$ goes to infinity. This is reminiscent of the annealing problem (see \cite{Royer} for the diffusion and \cite{MonmarchePDMP} for the velocity jump process) where $U_t$ is replaced by a deterministic $(\beta_t)_{t\geq 0}$, called the inverse temperature. It is classical that in this case, if $\beta$ increases faster than logarithmically then $X$ will eventually stay trapped forever in the cusp of a local minima. Yet, in our present case, as long as $X$ stays close to a positive minimum, $U$ increases linearly in time.

\bigskip

\noindent\textbf{Remarks :}	
\begin{enumerate}
	\item The particular form of the interacting potential $W(x,y) = F(x)F(y)$ implies that $W$ is a Mercer Kernel, which means the particle is repulsed by its past (see \cite{CEGMB}).
	
	If furthermore $\int_0^{2\pi}F(y)dy=0$, it has been shown in \cite[Theorem 2.13]{BR} that the normalized occupation measure $\frac{1}{t}\int_0^t \delta_{X_s}ds$ converges almost-surely to the uniform distribution on $\mathbb{S}^1$ whether or not $\mathcal{M}$ is empty in the weak self-interaction diffusion case. This is a major difference between strong and weak self-interaction.
	
	 We could also consider the case $W(x,y) = -F(x)F(y)$. Following the proof of Theorem \ref{ThmPrincipal}, it is not hard to see that in this case $X_t$ almost surely converges as $t$ goes to infinity to a point of $\mathcal M' = m(F,-)\cup M(F,+)$ which, as soon as $F$ is not constant and changes signs, is non-empty.
	\item If $F$ does not change signs, then, depending on the sign, $U_t$ converges either to $\infty$ or to $-\infty$ linearly fast. Therefore, Proposition \ref{PropHitting} and Proposition \ref{PropLocal} imply the almost-sure convergence of $X_t$ respectively either to a local minimum or to a local maximum of $F$.
\end{enumerate} 

%Our arguments are based on bounds for some hitting times of the processes (established in Section \ref{SectionHitting}) and on the density of their transition kernel (in Section \ref{SectionKernel}).   Section \ref{SectionLocalisation} is devoted to the case $\mathcal M \neq \emptyset$, while the case $\mathcal M = \emptyset$ is treated in Section \ref{SectionErgodique}.

We made the choice to write as much as possible notations, results and proofs which are common to both processes, isolating only the few  lemmas that deal with the specific technical difficulties of each case. Our arguments are based on bounds for some hitting times of the processes which are established in Section~\ref{SectionHitting}. From them we show in Section~\ref{SectionTemps} that, when $\mathcal M$ is empty,  the time for the processes to return to compact sets is short (i.e. it admits exponential moments). Section \ref{SectionDensite} is devoted to some uniform bounds of the transition kernel of the processes over compact sets, and Section \ref{SectionPreuve} to the proof of Theorem \ref{ThmPrincipal}.  

\textbf{Notation:} for $s\in\mathbb R$, $\lfloor s\rfloor = \max\{k\in\mathbb Z,\ k\leq s\}$ and $\lceil  s\rceil = \min\{k\in\mathbb Z,\ k\geq s\}$.

 %This paper is organised as follows. In section \ref{diffusion}, we consider the diffusion case *TODO*.
%%%%%%%%%%%%%%%%%%%%%%%%%%%%%%%%%%%%%%%%%%%%%%%%%%%%%%%%%%%%%%%%%%%%%%%%%%%%%%%%%%%%%%%%%%%%%%%%%%%%%%%%%%%%%%%%%%%%

\section{Hitting times}\label{SectionHitting} 

In this section, for a redaction purpose, we will hide the dependency on $U$ of the evolution of $X$. More precisely we will consider the (inhomogeneous in time) diffusion 
\begin{eqnarray}\label{EqDiffInhomo}
d X_t &=& dB_t - g(t) F'(X_t) dt
\end{eqnarray}
where $g$ is a Lipschitz-continuous function and, similarly, the inhomogeneous PDMP $(X,Y)$ with generator
\begin{eqnarray}\label{EqPDMPInhomo}
L_t f(x,y) &=& y \partial_x f(x,y) + \po \lambda + \po g(t) y F'(x) \pf_+\pf \po f(x,-y) - f(x,y)\pf
\end{eqnarray}
where the generator of an inhomogeneous Markov process $Z$ is by definition
\[ L_t f(z) = (\partial_s)_{| s=0}  \mathbb E\po f(Z_{t+s}\mid Z_t = z\pf. \]
Note the processes considered in Theorem \ref{ThmPrincipal} are particular cases of those defined here.

%In both cases we still let $U_t = U_0+ \int_0^t F(X_s) ds$.
Let $A=m(F,+)\cup m(F,-)$ be the set of minima of $F$, and $\delta \leq -\frac13\max\{ F(x)\; :\; x\in m(F,-)\} $ be positive and small enough so that
\begin{enumerate}
	\item[$\bullet$] for all $x\in A$, denoting by $I_x^\delta = [z_l,z_r]$ the connected component of $\{ F\leq F(x) + 2\delta\}$ containing $x$, then $F$ decreases on $[z_l,x]$ and increases on $[x,z_r]$.
	\item[$\bullet$] there exists $\kappa >0$ such that  for all $x\in A$ and $\eta\in [0,\delta]$,
			\begin{equation*}\label{preIm}
		d(x,B_x^{\eta})\geqslant \kappa \sqrt \eta,\; 
			\end{equation*} 
			where $ B_x^\eta = \{z\in I_x^\delta,\; F(z)=F(x)+\eta\}$.
\end{enumerate}% for all $x\in A$, denoting by $I_x^\delta = [z_l,z_r]$ the connected component of $\{ F\leq F(x) + 2\delta\}$ containing $x$, then $F$ decreases on $[z_l,x]$ and increases on $[x,z_r]$. Denoting by $ B_x^\eta = \{z\in I_x^\delta,\; F(z)=F(x)+\eta\}$, let  $\kappa >0$ be  such that  for all $x\in A$ and $\eta\in [0,\delta]$,
%		\begin{equation*}%\label{preIm}
%		d(x,B_x^{\eta})\geqslant \kappa \sqrt \eta,\; 
%		\end{equation*} 
The existence of $\kappa$ follows from the fact that $F(x) \leq F(x_0) + \| F''\|_\infty (x-x_0)^2/2$ for all $x\in \mathbb S^1$ and $x_0 \in A$. Remark that the definition of $\delta$ ensures that for all $x\in m(F,-)$, $\{ F\leq F(x) + 2\delta\} \subset \{ F \leq - \delta\}$. Finally, for all $\eta\in [0,\delta]$, let
\[B^\eta = \bigcup_{x\in A} B_x^\eta\hspace{20pt}\text{and}\hspace{20pt} C^{\eta} = \po \bigcup_{x\in A} I_x^\eta\pf^c. \]
In other words $C^\eta$ is the complementary of a neighbourhood of the minima of $F$ and $B^\eta$ is a set of intermediary points from $A$ to $C^\eta$. 
These sets (for $\eta=\delta$) are represented in Fig.~\ref{FigureABC}. Note that the choice of $\delta$ ensures that if $\mathcal M=\emptyset$ then $C^{\delta}$ contains $\{ F\geq - \delta\}$.

\begin{figure}
\centering
\includegraphics[scale=0.5]{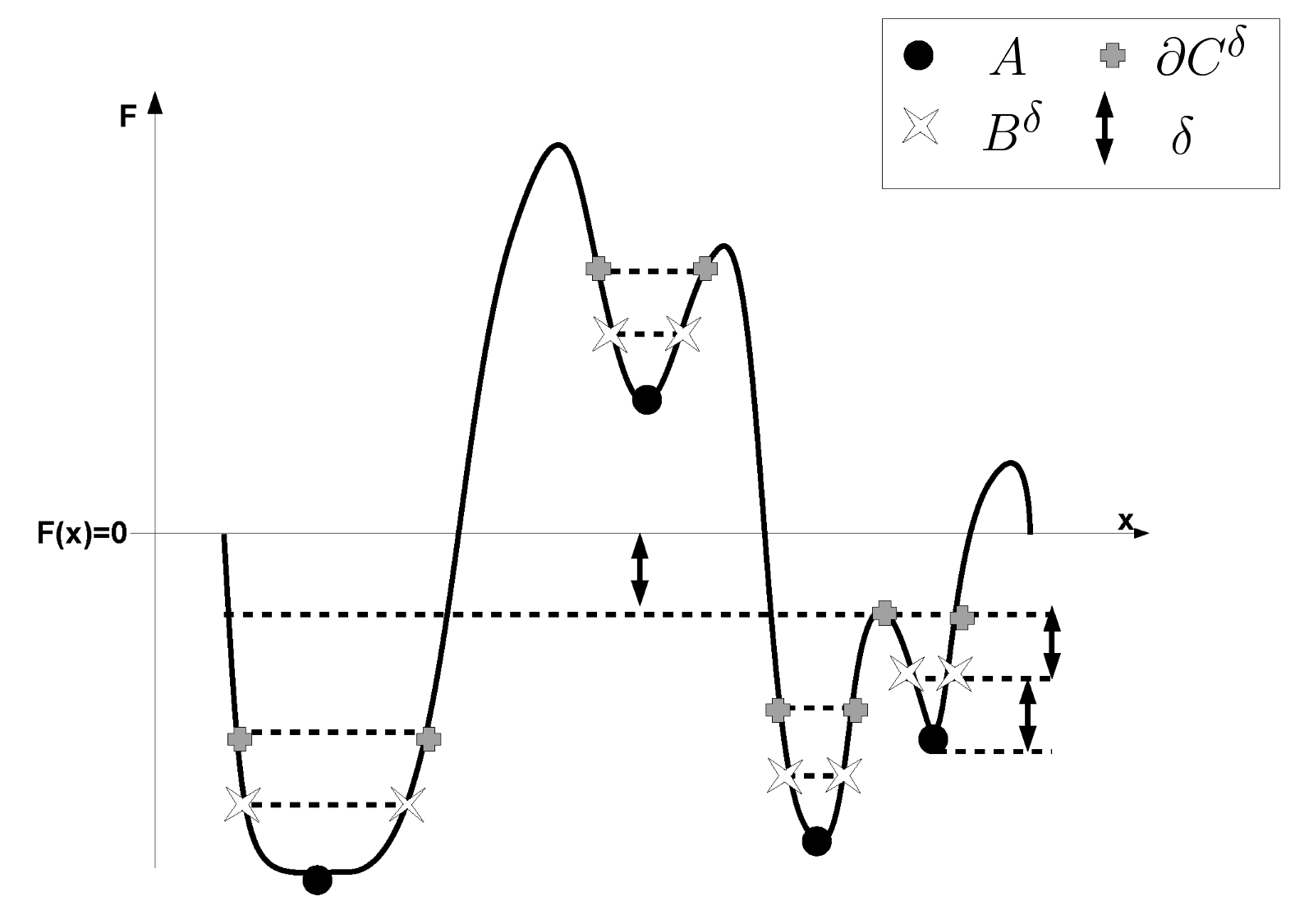}
\caption{Starting from a minimum in $A$, the process has to cross an intermediary point of $B^\delta$ halfway before reaching $C^\delta$. The energy level difference from $A$ to $B^\delta$, from $B^\delta$ to $C^\delta$ or (for a negative minimum) from $\partial C^\delta$ to $\{F=0\}$ is always at least $\delta$.}\label{FigureABC}
\end{figure}

For $x\in \mathbb S^1$ and $D\subset \mathbb S^1$ we write
\begin{eqnarray*}
T_{x\rightarrow D} &=& \inf\left\{ t\geq  0 \, :\,  X_t \in D\mid X_0 =x \right\}\\
q_{x\rightarrow D} &=& \mathbb P\po X \text{ reaches $D$ before $A$}\mid X_0 =x\pf
\end{eqnarray*}

For two real random variables $V,W$, recall that $V$ is said to be stochastically smaller than $W$, denoted by % which we write 
$V\overset{sto}\leq W$ (or equivalently $W\overset{sto}\geq V$), if for all $r\in\R$
\begin{eqnarray*}
	\mathbb P\po V>r\pf & \leq & \mathbb P\po W>r\pf.
\end{eqnarray*} 
If $V$ and $W$ have same law we write $V \overset{law}= W$.

The aim of this section is to prove the following:
\begin{prop}\label{PropHitting}
	There exist a constant $K > 1$ and non-negatives random variables $S$ and $R$ with an exponential moment   such that, for all $M>1$ and $\eta\in (0,\delta]$ with $M\eta>1$,  for all Lipschitz function $g\geq M$, if $X$ is defined by \eqref{EqDiffInhomo} or if $(X,Y)$ is defined by the generator \eqref{EqPDMPInhomo}, the following holds: 
%There exist nonnegative random variables $S$ and $R$ %on $\RR_+$  that 
%satisfying a Large Deviation Principle and a constant $K>1$ such that for all $\eta\in(0,\delta]$, for all $M>0$ and for all Lipschitz function $g\geq M$, if $X$ is defined by \eqref{EqDiffInhomo} or if $(X,Y)$ is defined by the generator \eqref{EqPDMPInhomo}, then the following holds:
\begin{eqnarray}
\forall x\in B^\eta,\hspace{32pt} q_{x\rightarrow C^\eta} & \leq & K M e^{-\eta M}\label{EqHittingProba}\\
\forall x\in A,\hspace{30pt} T_{x\rightarrow B^\eta} & \overset{sto}\geq & R  \eta\label{EqHittingR}\\
\forall x\in \mathbb S^1,\hspace{32pt} T_{x\rightarrow A} & \overset{sto}\leq & S .\label{EqHittingS}
\end{eqnarray}
\end{prop}
%\textbf{Remark:}
\bigskip

\noindent\textbf{Remark :} In the case of the velocity jump process $(X,Y)$, note that these bounds are uniform over the initial velocity $Y_0$.

\bigskip

The meaning of these bounds is the following. Suppose the auxiliary variable $U$ (whose role here is played by an arbitrary function $g$) stays for some time above a given level $M>1$. Then the position $X$ will fall in a local minima of $F$ within a time shorter than $S$, which does not depend on $M$ (i.e. a high $U$ can only accelerate the hitting time of $A$). Then to climb back up to an intermediary point of $B^{\eta}$, it takes a time $R \eta$, which is again uniform on $M>1$. From $B^{\eta}$, the probability to escape from the neighborhood of the minimum in one attempt (namely to reach $C^\eta$ before having fallen back to $A$, the bottom) is of order $e^{-\eta M}$, which is a classical metastability result (see \cite{Bovier,MonmarchePDMP} for instance) if $g$ is thought as an inverse temperature, since $\eta$ is the potential barrier to overcome.

The proof of Proposition \ref{PropHitting} is split in the two next subsections since the arguments are different for each dynamic. Note that in several proofs we will make assumptions like $x_1\leq x_2$ where $x_1$ and $x_2$ are in $\mathbb S^1$, which will make sense since at these times we will only be concerned by the behaviour of the processes on given simply connected intervals of $\mathbb S^1$.
%%%%%%%%%%%%%%%%%%%%%%%%
\subsection{For the diffusion}\label{SectionHittingDiff} 
%%%%%%%%%%%%%%%%%%%%%%%%%%%%
%Each of the assertions of Proposition \ref{PropHitting} is going to be the purpose of a lemma. We start with the following concentration result.
%\begin{lem} There exists $K>0$, such that for all $M>0$, $\eta\in(0,\delta]$ and %such that % There exists a constant $K\geq1$ %such that for all $M>0$, as soon as $m(t)\geq M$ for all $t\geq 0$, if $X$ solves \eqref{eqX}, 
%	such that for all $\eta\in(0,\delta]$, for all $M>0$ and 
% all Lipschitz function $g\geq M$%for all $M, t>0$ and as long as $U_t\geq M$:
%	\begin{equation*}
%	\forall x\in B^\eta,\qquad q_{x\rightarrow C^\eta}  \leq \mathbbm{1}_{\{M\eta < \ln(2)/4\}}+\mathbbm{1}_{\{M\eta\geqslant \ln(2)/4\}} KM e^{-2\eta M}.
%	\end{equation*} 
	%where $B=F^{-1}(-2\delta)$ and $C=\{x\in \mathbb{S}^1\mid F(x)\geqslant -\delta\}$ with $\delta =-\frac{1}{3}\max\Big(F(x)\mid x\in m(F,-)\Big)$.
%	\textcolor{red}{Reformulate the Lemma or try to obtain better approximation in the proof}
%\end{lem}
\begin{proof}[Proof of Inequality \eqref{EqHittingProba} in the diffusion case]
Let $M>1$ and $\eta \in[0,\delta]$ with $M\eta > 1$ be fixed and consider  the diffusion defined by \eqref{EqDiffInhomo} with $g\geq M$ and $X_0=x  \in B^\eta$. Let $x_0 \in A$ and $x_1\in  \partial C^\eta$ be such that %$x_1 < x <x_2$
$F$ is monotonous on the interval between $x_0$ and $x_1$ that contains $x$. In particular, $F(x_1)-F(x) = F(x)-F(x_0) = \eta$. Suppose without loss of generality that $x_0 < x < x_1$.   Since $g\geq M$, it follows from Ikeda-Watanabe's comparison result \cite[Theorem 1.1, Chapter VI]{IW} that
\begin{equation*}
q_{x\rightarrow C^\eta} \leqslant \mathbb{P}\po \widetilde X \text{ hits } C^\eta \text{ before } A\mid \widetilde X_0=x\pf := \tilde q_{x\rightarrow C^\eta} ,
\end{equation*} 
where $\widetilde X$ solves the SDE
\begin{equation*}
d \widetilde X_t = dB_t -MF'\po \widetilde X_t\pf dt.
\end{equation*}
Indeed, note that the definitions of $q_{x\rightarrow C^\eta} $ and  $\tilde q_{x\rightarrow C^\eta}$ only involve the processes on the interval $[x_0,x_1]$, on which $F'\geq 0$. The scale function of $\widetilde X_t)_{t\geqslant 0}$ is defined by
\begin{eqnarray*}
p(y)&=&\int_{x }^y\exp \left(-2\int_{x  }^z -M F'(s)ds\right )dz \\
&=& \int_{x  }^y e^{2M (F(z)-F(x ))}dz.
\end{eqnarray*}
By \cite[Proposition 5.22, Chapter 5.5]{K-S},% Proposition *** in \emph{(Revuz-Yor or Karatzas-Shreve ; Rmq : dans Malrieu c'est plutôt la prop 4.3.4 qu'il faut regarder, où il n'y pas d'histoire de limite en temps puisqu'on reste dans l'intervalle où tout est bien défini)},
\begin{equation*}
 \tilde q_{x\rightarrow C^\eta} =  \frac{p(x_1)-p(x)}{p(x_1)-p(x_0)}  \leq \frac{2\pi e^{2M\eta }}{\int_{x_0}^{x_1} e^{2M\po F(z) - F(x)\pf}dz}
\end{equation*}
 where we used the local monotony of $F$.  On the other hand,  
\begin{eqnarray*}
\int_{x_0}^{x_1} e^{2M\po F(z) - F(x)\pf}dz  
& \geqslant &  \frac{1}{2M \Vert F'\Vert_\infty} \int_{x_0}^{x_1} 2M F'(z) e^{2M (F(z)-F(x))}dz\notag\\
&=& \frac{1}{2M \Vert F'\Vert_\infty}\po e^{4M\eta }-1\pf.
\end{eqnarray*}
%\textcolor{red}{On arrive à avoir $G(x_2)\geqslant \frac{e^{2M(F(x_2)-F(0))}-1}{2M \Vert F'\Vert_\infty}$}
Therefore, as $M\eta >1$,
\[ q_{x\rightarrow C^\eta}\ \leqslant\   4\pi M \Vert F'\Vert_\infty \frac{e^{2M\eta}}{e^{4M\eta}-1} 
\  \leqslant\   8\pi M \Vert F'\Vert_\infty e^{-2M\eta}. \]

%\textcolor{red}{Avec la majoration ci-dessus, on a $q_{x\rightarrow C^\eta}\leqslant 4\pi M \Vert F'\Vert_\infty \frac{e^{2M\eta}}{e^{4M\eta}-1} \leqslant 8\pi M \Vert F'\Vert_\infty e^{-2M\eta} $ lorsque $M\eta\geqslant \ln(2)/4.$ }
%if $\eta$ is chosen such that $M\eta\geqslant \ln(2)/4.$ \textcolor{red}{D'un autre c\^oté, lorsque $x_2$ est très petit, $\frac{p(x)-P(0^+)}{p(x_2^-)-p(0^+)}\simeq \frac{x}{x_2}e^{2M (F(x)-F(x_2))}\leqslant e^{-2M\eta}$. De plus, $G$ est convexe sur $[0,x_2]$.} 
\end{proof}
To prove the two other assertions of Proposition \ref{PropHitting}, we need the following comparison result: 
\begin{lem}\label{compar} Let $x_0$ be a local extrema of $F$ and   $\varepsilon >0$ be such that $F'$ is monotonous on $J_\varepsilon := (x_0 -\varepsilon ,x_0 +\varepsilon)$. Consider $X$ the diffusion defined by \eqref{EqDiffInhomo}, with $g\geq M>1$, starting at $X_0=x\in J_\varepsilon$, and  let $W$ be  a standard Brownian motion. Denote by
\[\chi^\varepsilon (x) =\inf \left\{ t>0 \, :\   X_t\notin J_\varepsilon \right\}\hspace{15pt}\text{and}\hspace{15pt}\iota^\varepsilon =\inf \left\{ t>0 \, :\   |W_t|=\varepsilon \right\}\] 
%\begin{equation}\label{tempsSortieBM}
%\iota^\varepsilon =\inf \left\{ t>0,\ \vert B_t\vert =\varepsilon \right\}
%\end{equation}
the respective exit time  from $J_\varepsilon$ of $X$ and $x_0+W$. Then:
\begin{enumerate}
\item if $x_0$ is a local maximum of $F$, for all $x\in J_\varepsilon$,
\begin{equation*}
\chi^\varepsilon (x)\overset{sto}\leqslant\iota^{\varepsilon}.
\end{equation*}
\item if $x_0$ is a local minimum of $F$,
\begin{equation*}
\chi^\varepsilon (x_0)\overset{sto}\geqslant \iota^{\varepsilon}.
\end{equation*}
\end{enumerate}
\end{lem}
\begin{proof} First, note that by symmetry the exit time from $J_\varepsilon$ of $x+W$ has the same law as the exit time  of $x+2(x_0-x) + W$, and since the process $x_0+W$ necessarily crosses $x$ or $x+2(x_0-x)$ before leaving $J_\varepsilon$, the exit time of $x_0+W$ is stochastically greater than the one of $x+W$ for any $x\in J_\varepsilon$.

Consider  $\Theta_t = (X_t-x_0)^2$, which solves% and $Z$ be the processes defined by $\Theta_0=Z_0=(X_0-x_0)^2\in [0,\varepsilon^2)$ and
	\begin{eqnarray*}
		d\Theta_t &=& 2\sqrt{\Theta_t}d\widetilde B_t +dt-2g(t)\big((X_t-x_0)F'(X_t) \big)dt,%\\
%	dZ_t& =& 2\sqrt{Z}_t dW_t + dt.
	\end{eqnarray*} 
 where $\widetilde B_t = \int_0^t sign(X_s-x_0)d B_s$ is still a standard Brownian motion.   As easily seen, $ \po X_0-x_0 + \widetilde B\pf^2$ is a weak solution of  %solves
 	\begin{eqnarray*}
		dZ_t &=& 2\sqrt{Z_t}d\widetilde B_t +dt  .
	\end{eqnarray*}
	 Now, we can compare processes $(\Theta_t)_{t\geq 0}$ and $(Z_t)_{t\geq 0}$. When $x_0$ is a maximum (resp. minimum) of $F$, $(x-x_0)F'(x)$ is non-positive (resp. non-negative) on $J_\varepsilon$, so that by Ikeda-Watanabe's comparison result  (\cite[Theorem 1.1, Chapter VI]{IW}), $\Theta_t\geqslant Z_t$ (resp. $\Theta_t\leqslant Z_t$)  up to the first time where $\Theta$ reaches $\varepsilon^2$. As a conclusion, when $x_0$ is a maximum, $\Theta$ reaches $\varepsilon^2$ before $Z$, and thus in a time stochastically smaller than $\iota^\varepsilon$, whereas when $x_0$ is a minimum, $\Theta$ reaches $\varepsilon^2$ after $Z$ and the latter happens at a time with law $\iota^\varepsilon$ if the starting point is $x_0$.
\end{proof}

\begin{proof}[Proof of Inequality \eqref{EqHittingR} in the diffusion case]
Recall that there exists a constant $\kappa>0$ such that for all $x\in A$ and $\eta<\delta$, $d(x,B_x^{\eta})\geqslant \kappa \sqrt \eta$. From Lemma~\ref{compar} and the Brownian motion's scaling property,
\begin{equation*}
T_{x\rightarrow B^\eta}\geqslant \chi^{\kappa \sqrt \eta } ( x)\overset{sto}{\geqslant} \iota^{\kappa \sqrt \eta}\overset{law}{=} \eta \iota^\kappa .
\end{equation*}
The fact that $\iota^\kappa$ has an exponential moment is a consequence of \cite[Theorem 2]{CT}. 
%\textcolor{red}{Et Rmq moments expo. cf temps de sortie d'un Bessel de dim 1}
\end{proof}
 
\begin{proof}[Proof of Inequality \eqref{EqHittingS} in the diffusion case]
For a given small enough $\varepsilon>0$, denote by
\[E^\varepsilon = \underset{x\in  M(F,+)\cup M(F,-)}{\bigcup}  (x-\varepsilon,x+\varepsilon)\]
 the set of points which are at a distance less than $\varepsilon$ from a maximum of $F$. Let $X$ be the diffusion defined by \eqref{EqDiffInhomo} with $g\geq M$. We apply the following procedure:
\begin{enumerate}
\item If, at some time, $X_t \in E^\varepsilon$, wait until it leaves $E^\varepsilon$, which according to the first part of Lemma \ref{compar} happens in a time stochastically smaller than $\iota^\varepsilon$.
\item If at some time $t_0$, $X$ leaves $E^\varepsilon$, compare it with $X_{t_0}+B$ where $B$ is the Brownian motion that drives the SDE \eqref{EqDiffInhomo}. More precisely by Ikeda-Watanabe's comparison result, $F(X_t)\leqslant F(X_{t_0}+B_t)$ up to the time where either $X$ or $X_{t_0}+B$ reach an extremum of $F$. 
\item Wait until $B$ reaches an extrema of $F$. If this is a maximum, go back to the first step. If this is a minimum then necessarily, at this time, $X$ has already crossed this minimum, stop the procedure.
\end{enumerate}
Note that, $\varepsilon$ being fixed, the probability that $x_0 + B$ reaches a maximum rather than a minimum is bounded above by some $p<1$ which is uniform over all $x_0\in \partial E_{\varepsilon}$. Hence the number of iteration of the procedure is stochastically less than a geometric random variable $G$ with parameter $p$.  Conditionally to whether the Brownian motion reaches a minimum or a maximum in step 3, the law of the duration of the third step is different, but in either cases it is stochastically smaller than $\iota^{2\pi}$. Therefore the total duration of one iteration of the procedure is stochastically smaller than $\iota^{C}$ for some constant $C>0$, independently from whether this is the last iteration or not.  Let $(\iota_k)_{k\geq 0}$ be i.i.d copies of $\iota^{C}$, independent from $G$.
	
	We have obtained that for all $x\in\mathbb{S}^1$,
	\begin{eqnarray*}
	T_{x\rightarrow A} & \overset{sto}{\leqslant} & \sum_{k=0}^G \iota_k
	\end{eqnarray*} 
so that
\begin{eqnarray*}
\mathbb E\po e^{c T_{x\rightarrow A}}\pf & \leq & \mathbb E \po \po \mathbb E \po e^{c \iota_0}\pf\pf^{G} \pf 
\end{eqnarray*}
which is finite for $c$ small enough since $G$ admits an exponential moment.
\end{proof}
%%%%%%%%%%%%%%%%%%%%%%%%%%%%
\subsection{For the velocity jump process}\label{SectionHittingPDMP} 
%%%%%%%%%%%%%%%%%%%%%%%%%%%%
This subsection is devoted to the proof of Proposition \ref{PropHitting} in the PDMP case, namely for the inhomogeneous Markov process $(X,Y)$ with generator \eqref{EqPDMPInhomo}. First we construct a trajectory of the process $(X,Y)$ in the following way: consider two independent i.i.d. sequences of standard (with mean 1) exponential random variables $(E_i)_{i\in\mathbb N}$ and $(F_i)_{i\in\mathbb N}$. Set $T_0=0$ and suppose the process has been defined up to some time $T_k$ independently from $(E_i,F_i)_{i\geq k}$. Let
%Starting at time $t$ from $(x_t,y_t)$, we set $X_s = x_t + (s-t) y_t$ and $Y_s=y_t$ up to time $t+\theta_1\wedge \theta_2$ where
\begin{eqnarray*}
	\theta_1 &=& \inf\left\{t>0 \, :\,  \int_0^t g(T_k+s)(Y_{T_k} F'(X_{T_k} + s Y_{T_k}))_+ \dd s >E_k\right\},\\
	\theta_2 &=& \frac1\lambda F_k,
\end{eqnarray*}
and $T_{k+1}=T_k+\theta_1\wedge \theta_2$, which is the next jump time. If $T_{k+1}=T_k+\theta_1$ we say that the jump is due to the landscape, else we say it is due to the constant rate $\lambda$. In either cases, set $X_t = X_{T_k} + (t-T_k) Y_{T_k}$ for all $t\in [T_k,T_{k+1}]$, $Y_t=Y_{T_k}$ for all $t\in [T_k,T_{k+1})$ and $Y_{T_{k+1}} = - Y_{T_k}$. Thus by induction the process is defined up to time $T_{n}$ for all $n$. Note that even if, depending on $g$, the rate of jump may not be bounded, two jumps due to the landscape cannot be arbitrarily close (since at such a jump time, $y F'(x)$ becomes non-positive), so that there cannot be infinitely many jumps in a finite time and $T_n \rightarrow \infty$ as $n\rightarrow\infty$.

\begin{proof}[Proof of Inequality \eqref{EqHittingProba} in the PDMP case]
	We mainly have to adapt to our inhomogeneous setting the proof of  \cite[Proposition 4.1]{MonmarchePDMP}. Without loss of generality, we consider the following configuration: $x_0\in A,\; x_1\in B^\eta$ and $x_2\in \partial C^\eta$ with $x_0<x_1<x_2$, and $F$ is increasing on $[x_0,x_2]$.

	Let $M\geqslant 1$ and   $\mathcal L_M$ be the set of Lipschitz functions $g \geqslant M$. For all $x\in [x_1,x_2]$, set
\begin{eqnarray*}
\eta_x &= &\underset{g\in \mathcal L_M}\sup \mathbb{P}((X,Y) \text{ reaches $(x_2,1)$ before $(x,-1)$ }\mid (X_0 ,Y_0)=(x,1)).
\end{eqnarray*}
 where the supremum runs over the function $g$ that appears in the generator \eqref{EqPDMPInhomo} of the process $(X,Y)$.

Consider a process $(X,Y)$ with generator \eqref{EqPDMPInhomo} with some function $g\in\mathcal L_M$. For a small $\varepsilon >0$, suppose that $(X_0,Y_0)=(x-\varepsilon ,1)$.

Then the probability that $X$ goes from $x-\varepsilon$ to $x$ without any jump is less than $1-\varepsilon\po  M F'(x) +\lambda\pf + \underset{\varepsilon\rightarrow 0}o(\varepsilon)$ and the probability it reaches $(x,1)$ before $(x-\varepsilon,-1)$ but with at least one jump %(and so with at least two jumps)
 is of order $\varepsilon^2$ as $\varepsilon\rightarrow 0$ (uniformly over $g\in \mathcal L_M$). 
 
If the process has reached $(x,1)$, it has a probability less than $ \eta_x $ to reach $(x_2,1)$ before having fallen back to $(x,-1)$. Nevertheless, if indeed it has fallen back to $(x,-1)$, it has a probability $\varepsilon\lambda  + \underset{\varepsilon\rightarrow 0}o(\varepsilon)$ to jump before reaching $(x-\varepsilon,-1)$, in which case it reaches again $(x,1)$ with probability $1+\underset{\varepsilon\rightarrow 0}o(1)$. In this latter case, it reaches $(x_2,1)$ before $(x-\varepsilon,-1)$ with probability less than  $\eta_x + \underset{\varepsilon\rightarrow 0}o(1)$. Thus everything boils down to
	\begin{eqnarray*}
\eta_{x-\varepsilon} & \leqslant  & \po 1-\varepsilon\po MF'(x)+\lambda \pf\pf  \eta_x \po 1 + \varepsilon \lambda\pf + \underset{\varepsilon\rightarrow 0}o(\varepsilon)\\
		& = & \po 1-\varepsilon MF'(x)\pf  \eta_x  + \underset{\varepsilon\rightarrow 0}o(\varepsilon).
	\end{eqnarray*}
%Since the negligible terms are uniform with respect to the time%$m$, hence $t$,
%\begin{eqnarray*}
%\eta_{x-\varepsilon & \leq  &   \po 1-\varepsilon MF'(x)\pf  \eta_x  + \underset{\varepsilon\rightarrow 0}o(\varepsilon)
%\end{eqnarray*}
%which 
Together with $\eta_{x_2} = 1$, it yields $\eta_x\leqslant e^{-M\po F(x_2)-F(x)\pf} $, and in particular $\eta_{x_1} \leqslant e^{-\eta M}$.

Let
\begin{eqnarray*}
r_y &= &\underset{g\in \mathcal L_M}\sup \mathbb{P}((X,Y) \text{ reaches $(x_2,1)$ before $(x_0,-1)$ }\mid (X_0 ,Y_0)=(x_1,y)).
\end{eqnarray*}
Starting from $(x_1,-1)$ and until the process either jumps or reaches $(x_0,-1)$, we have $YF'(X)<0$ so that, 	whatever the function $g$ in \eqref{EqPDMPInhomo} is, there cannot be any jump due to the landscape during this time. On the other hand if $\theta_2>2\pi$, which happens with probability $e^{-2\lambda \pi}$, there is also no jump due to the constant rate during this time, so that
		\begin{eqnarray*}
	\mathbb P\po  (X,Y)\text{ reaches }(x_1,1)\text{ before }(x_0,-1) \mid (X_0 ,Y_0) = (x_1,-1)\pf
& \leqslant &	1- e^{-2\lambda \pi}.
	\end{eqnarray*}
On the one hand it means that $r_{-1} \leq \po 1- e^{-2\lambda \pi}\pf r_1$ and on the other hand that
	\begin{eqnarray*}
r_1 &\leqslant & \eta_{x_1} + \po 1- e^{-2\lambda \pi}\pf r_1
	\end{eqnarray*}	
and finally that 
\[ q_{x_1\rightarrow C^\eta} \ \leqslant\ \max(r_1,r_{-1})\ \leqslant\ e^{2\lambda \pi} e^{-\eta M}. \]
\end{proof}

\begin{proof}[Proof of Inequality \eqref{EqHittingR} in the PDMP case]
	Since $\vert Y\vert =1$, the time needed to reach $B^\eta$ from $A$ is deterministically larger than $d(A,B^\eta)\geqslant \kappa \sqrt \eta $.
\end{proof} 
%Finally, wherever on the circle the particle is, the time needed 

\begin{proof}[Proof of Inequality \eqref{EqHittingS} in the PDMP case]
Suppose that, at some point in the construction of a trajectory, $\theta_2> 4\pi$, which happens with probability $e^{-4\lambda\pi}$. If there is also no jump due to the landscape in the meanwhile, $X$ covers the whole circle and in particular reaches $A$ in a time less than $2\pi$. On the other hand if there is a jump due to the landscape before time $2\pi$, the velocity turns to its opposite, and from then and up to the hitting time of $A$, $YF'(X)<0$, so that in the meanwhile there cannot be another jump due to the landscape: $A$ is attained in a time less than $4\pi$.

It means that as soon as $\theta_2> 4\pi$, $X$ reaches $A$ in a time less than $4\pi$, so that starting from any point of $\mathbb S^1$, $X$ reaches $A$ in a time stochastically smaller than $4\pi G$ where $G$ is a geometric variable with parameter $e^{-4\lambda\pi}$.
\end{proof}
%%%%%%%%%%%%%%%%%%%%%%%%%%%%%%%%%%%%%%%%%%%%%%%%%%%%%%
\section{Stability}\label{SectionTemps}
%%%%%%%%%%%%%%%%%%%%%%%%%%%%%%%%%%%%%%%%%%%%%%%%%%%%%%
In this section we consider either $Z=(X,U)$ or $Z=(X,U,Y)$ such as in Theorem \ref{ThmPrincipal}, and we are interested in the time of return of $Z$ to compact sets. For $M>1$ we write
\begin{eqnarray*}
	\tau_M &=&\inf\{t>0 \, :\ |U_t|\leqslant M\}
\end{eqnarray*}
and we aim to prove $\tau_M$ admits exponential moments when $\mathcal{M}=\emptyset$. For this purpose, we are going to establish that, for some $\kappa>0$, $V(u) = \exp(\kappa u)$ is a so-called Lyapunov function for both processes, which classically implies the latter.

The notations of Section \ref{SectionHitting} are kept, in particular the constant $\delta$, and the constant $K$ and the random variables $R, S$ appearing along this section are those given by Proposition \ref{PropHitting}.
\begin{lem}\label{LemBorneGlobal}
	Suppose $m(F,+)=\emptyset$. Let $M>1$ be such that $KMe^{-\delta M}<1 $, and let $(S_i)_{i\in\mathbb N}$, $(R_i)_{i\in\mathbb N}$ and $(G_i)_{i\in\mathbb N}$ be independent i.i.d. sequences where $S_0$ (resp. $R_0$) is a copy of $S$ (resp. $ \delta R$) and $G_0$ has geometric law with parameter $KMe^{-\delta M}$. For $t\geq 0$ let 
	\begin{eqnarray*}
		N_t &=& \inf\left\{ n\in\mathbb N \, : \ \sum_{k=1}^{G_0+\dots+G_n} R_k \geq t\right\}.
	\end{eqnarray*}
	Then for all $t>0$ and for any initial condition $Z_0$ with $U_0> M$,
	\begin{eqnarray*}
		\int_0^{t\wedge \tau_M} \mathbbm 1_{\{F(X_s)\geq -\delta\}} \dd s & \overset{sto}\leq & \sum_{k=0}^{N_t} S_k.
	\end{eqnarray*}
\end{lem}
\begin{proof}%\textcolor{red}{reprendre avec les bonnes notations}
While $t\leqslant \tau_M$, the estimates of Proposition \ref{PropHitting} hold for $X$. In particular, independently from its initial condition, the process reaches $A$ in a time stochastically smaller than $S_0$. Then it takes at least a time $R_1$ to climb back to $B^\delta$. From there, it reaches $C^\delta$ with probability less that $ KM e^{-\delta M}$, else it falls back to $A$. Therefore it remains a time  stochastically greater than $ \sum_{k=1}^{G_0} R_k$ in $(C^\delta)^c \subset \{ F\leqslant -\delta \}$ before reaching $C^\delta$.  When this finally occurs, the process falls again back to $A$ after a time less than $S_1$ (independently from what occurred before it had reached $C$). We call this an excursion in $C^\delta$. After $n$ excursions, the process has stayed at least a time $ \sum_{k=1}^{G_0+\dots+G_n} R_k$ in $\{ F\leqslant -\delta \}$, which implies in particular that at time $t$ there have been stochastically less  than $N_t$ excursions. Thus during a time $t$, the time spent in $C^\delta$ is stochastically less than $\sum_{k=0}^{N_t} S_k$.
\end{proof}

\begin{prop}\label{PropLyapunov}
If $\mathcal{M}=\emptyset$, there exists $\kappa_0>0$ such that, for all $\kappa \in (0,\kappa_0]$, there exists $t_0$ such that, for all $t\geqslant t_0$, there exists $C_t\geqslant 0$ such that for all initial conditions,
\begin{eqnarray*}
\mathbb E\po e^{\kappa |U_t|}\pf & \leqslant & \frac12 e^{\kappa |u_0|} + C_t.
\end{eqnarray*}
\end{prop}
Note that, in the sequel, the dependency of $C_t$ with respect to $t$ will not matter.
\begin{proof}
Let $\kappa_0>0$ be small enough so that $\mathbb E\po e^{\kappa_0 (\delta + \max| F|)S} \pf < \infty$ and, for $\kappa\in(0,\kappa_0]$, let $t_0$ be such that $\mathbb E\po e^{ (\kappa (\delta + \max| F|)S} \pf \leq  e^{\kappa \delta t_0}/4$. Finally, let $M>1$ be such that $KMe^{-\delta M}<1 $ and fix $t\geq t_0$.

Let $u_0>M+t \max |F|$, so that $\tau_M \geqslant t$ almost surely. Hence, Lemma~\ref{LemBorneGlobal} yields
\begin{eqnarray*}
|U_t| & \leq & u_0 - \delta t + \po\delta + \max F\pf  \int_0^t \mathbbm 1_{\{F(X_s)\geq -\delta\}} \dd s\\
& \overset{sto}\leq & u_0 - \delta t + \po\delta + \max F\pf \min\po t, \sum_{k=0}^{N_t} S_k\pf.
\end{eqnarray*}
Hence, distinguishing the cases $N_t = 0$ and $N_t>0$,
\begin{eqnarray*}
\mathbb E\po e^{\kappa \po |U_t|-u_0\pf}\pf & \leqslant & e^{-\kappa \delta t} \mathbb E\po e^{\kappa (\delta + \max F)S_0}\pf + e^{\kappa t\max F} \mathbb P \po N_t > 0\pf.
\end{eqnarray*}

Now, for all $a\in\mathbb N^*$, 
\begin{eqnarray*}
	\Big\{ G_0 \geq  a \hspace{20pt}\text{and}\hspace{20pt}  \sum_{i=1}^a R_i \geq t \Big\} & \subset & \{N_t = 0\},
\end{eqnarray*}
so that, considering the complementary sets,
\begin{eqnarray*}
 \mathbb P \po N_t > 0\pf & \leq &  \mathbb P \po G_0 < a\pf +  \mathbb P \po \sum_{k=1}^{a} R_k < t\pf.
\end{eqnarray*}
Applied with $a = \lceil M\rceil$, this reads
\begin{eqnarray*}
 \mathbb P \po N_t > 0\pf & \leq & 1- (1-K M e^{-\delta M})^M+  \mathbb P \po \sum_{k=1}^{\lceil M\rceil} R_k < t\pf \, \underset{M\rightarrow +\infty}\longrightarrow\, 0.
\end{eqnarray*}
In particular, for $M$ large enough,  $\mathbb P \po N_t > 0\pf  \leq \exp(-\kappa t\max F)/4$. Then we have proved that there exists $M>0$ such that for all $u_0>M+t \max |F|$,
\begin{eqnarray}\label{EqLyapounov1}
\mathbb E\po e^{\kappa |U_t|}\pf & \leqslant & \frac12 e^{\kappa |u_0|}.
\end{eqnarray}
The case $u_0 < - M - t\max |F|$ is similar (by changing $U$ and $F$ to their opposites), so that there exists $M>0$ such that  \eqref{EqLyapounov1} holds for all $u_0$ with $|u_0|>M+t \max |F|$. Finally, if $|u_0|\leqslant M+t \max |F|$ then
\begin{eqnarray*}
\mathbb E\po e^{\kappa |U_t|}\pf & \leqslant & e^{\kappa (M+2t\max |F|)}\ := \ C_t,
\end{eqnarray*} 
which concludes.
\end{proof}

In fact, we will also need this similar result, obtained by the same arguments:
\begin{lem}\label{LemmAtteinte}
If $m(F,+)=\emptyset$ then, for $M$ large enough, for all initial condition $u_0\geq M$,  $\inf\{s\geq 0 \ : \ U_s \leq M\}$ is almost surely finite and admits an exponential moment.
\end{lem}
\begin{proof}
In the proof of Proposition \ref{PropLyapunov} we have in fact obtained that,  under the  assumption that $m(F,+)=\emptyset$, there exists $\kappa,t,C_t>0$ such that for any initial condition $u_0\geq 0$,
\begin{eqnarray*}
\mathbb E \po e^{\kappa U_t}\pf & \leq & \frac12 e^{\kappa u_0} + C_t.
\end{eqnarray*}
Denoting $n_0= \inf\{n\in\mathbb N,\ \exp(\kappa U_{nt}) \leq 4C_t\}$, then the random sequence $\{(4/3)^{(n\wedge n_0)}\exp(\kappa U_{(n\wedge n_0)t})\}_{n\geq 0}$ is a submartingale. As a consequence, by Fatou's lemma, $\mathbb E\po (4/3)^{n_0}\pf $ is finite, which concludes.
%
%Symmetrically (changing $u_0$ and $F$ to their opposites), the following holds: if  $M(F,-)=\emptyset$, then for $M$ large enough, if $u_0<0$, then $\tau_M$ is almost surely finite and admits an exponential moment. 
\end{proof}

%%%%%%%%%%%%%%%%%%%%%%%%%%%%%%%%%%%%%%%%%%%
\section{Transition kernel bounds}\label{SectionDensite}
%%%%%%%%%%%%%%%%%%%%%%%%%%%%%%%%%%%%%%%%%%%%%
In this section we still consider either $Z=(X,U)$ or $Z=(X,U,Y)$ such as in Theorem~\ref{ThmPrincipal}, and we call $E$ its state space, namely either $\mathbb S^1\times \RR_+$ or $\mathbb S^1\times \RR_+\times\{-1,1\}$. %Under the assumption $\mathcal{M}=\emptyset$,
We aim to prove that  the following local Doeblin condition holds:
\begin{prop}\label{PropNoyau}
Let $\mathcal K$ be a compact set of $E$. There exists $t_0>0$ such that, for all $t\geq t_0$, there exist $0<c<1$ and a probability measure $\nu$ on $E$ such that for all $z\in\mathcal K$, for all Borel set $D$,
\begin{eqnarray*}
\mathbb P\po Z_{t} \in D\mid Z_0 = z\pf &\geq & c \nu(D).
\end{eqnarray*}
\end{prop}
This condition classically ensures that two processes starting at different states in a compact set $\mathcal K$ can be coupled after a time $t$ with some probability $c>0$. Together with the Lyapunov condition obtained in Proposition \ref{PropLyapunov} which implies that, starting away, the processes reaches $\mathcal K$ in a short time, this Doeblin condition ensures exponentially fast mixing for the process  (see e.g. \cite{Hai}).

For the diffusion process, this classically follows from an hypoellipticy argument. By contrast, note that the velocity jump process is not regularizing, in the sense its transition kernel is never absolutely continuous with respect to the Lebesgue measure (at all time there is a positive probability that the process hasn't jumped yet). However the Doeblin condition can still be obtained from some controllability property and a partial regularization. 

Since, again, the arguments are different for both processes, we split the proof of Proposition \ref{PropNoyau} in two paragraphs.
%%%%%%%%%%%%%%
\subsection{For the diffusion}
%%%%%%%%%%%
In this subsection we consider  the process $Z=(X,U)$ induced by the generator~(\ref{MODELEdiff}), namely the solution of the SDE
  \begin{eqnarray}\label{EqOrigin31}
& &  \left\{\begin{array}{rcl}
 d X_t &=& d B_t - U_t F'(X_t) dt\\
  d U_t &=& F(X_t) dt.
\end{array}  \right.
  \end{eqnarray}
%In this subsection, we denote by $P_t((x_0,u_0),dxdu)$ the induced transition kernel with initial condition $(x_0,u_0)$ and $P_t$ the associated semi-group. At first glance, due to the degeneracy of \eqref{EqOrigin31}, it is far from being clear that $P_t((x_0,u_0),dxdu)$ has a density with respect to the Lebesgue measure. However, it is the case as proven in the next proposition.
\begin{lem}\label{densiteDiff}
 For all $z_0\in\mathbb{S}^1\times \RR$ and $t>0$, the transition kernel $\mathbb P\po Z_t\in \cdot\ |\ Z_0=z_0\pf$ admits a smooth density with respect to the Lebesgue measure %, with moreover a compact support. 
 and its support is $\mathbb{S}^1\times [u_0 +(\min F) t, u_0 +(\max F) t]$.
\end{lem}
  \begin{proof}
  	For $(x,u)\in\mathbb{S}^1\times \RR$, set 
  	\begin{equation*}
G_0(x,u)=  \begin{pmatrix}
1 \\
0
\end{pmatrix}\text{ and } G_1(x,u)= \begin{pmatrix}
-uF'(x) \\
F(x)
\end{pmatrix}.
  	\end{equation*}
  	Denoting $\nabla G_i$ the Jacobian matrix of $G_i$ for $i=1,2$, the Lie-bracket of $G_0,G_1$ is the vector field $[G_0,G_1]$ defined  by 
  	\begin{equation*}
  	[G_0,G_1](x,u)= G_0(x,u)\nabla G_1(x,u)- G_1(x,u)\nabla  G_0(x,u)\,,
  	\end{equation*}
  	which is here equal to
  	\[\partial_x G_1(x,u)= \begin{pmatrix}
  	-uF''(x) \\
  	F'(x)
  	\end{pmatrix}.\]
  	By induction, replacing $F$ by $F^{(k)}$ for  $k\in\mathbb N$ in the previous computation, we get
  	  \begin{equation*}
  	  [\underbrace{G_0,[G_0,\ldots[G_0}_{k\text{ times}},G_1]\ldots]](x,u)= \partial_x^{(k)} G_1(x,u)= \begin{pmatrix}
  	  -uF^{(k+1)}(x) \\
  	  F^{(k)}(x)
  	  \end{pmatrix}.
  	  \end{equation*}
  	Therefore, by our non-degeneracy assumption on $F$ (namely that for all $x\in\mathbb S^1$, $F^k(x)\neq 0$ for some $k$), the SDE \eqref{EqOrigin31} satisfies everywhere the H\"ormander condition (see for instance \cite{Hai}), which gives the first part of the lemma. %, while the second one ensues from
  	For the second part, first note that for $z=(x_0,u_0)\in\mathbb{S}^1\times \RR$,
  	\begin{equation*}
  	\Big((X_t,U_t)\Big)_{t\geqslant 0}\subset \mathbb{S}^1\times [u_0 +(\min F) t, u_0 +(\max F) t].
  	\end{equation*}
  	In order to apply the Stroock-Varadhan support Theorem \cite[Theorem 5.2]{SV}, we are lead to the study of the following deterministic control problem. Denote by $(x_s, u_s)_{s\geqslant 0}$   the solution of the ordinary differential equation
  	  \begin{eqnarray}\label{EDOxu}
  	  & &  \left\{\begin{array}{rcl}
  	  \dot{x} &=& v(t) - u F'(x) \\
  	  \dot{u} &=& F(x) 
  	  \end{array}  \right.
  	  \end{eqnarray}
  	  with initial condition $(x(0), u(0))=z$ and where $s\mapsto v(s)$ is a piecewise constant function. The proof will be concluded if, given any  $z'=(x',u')\in \mathbb{S}^1\times \Big(u_0 +(\min F) t, u_0 +(\max F) t\Big)$, we can build a  function $v$ such that $(x(t), u(t))$ is arbitrarily close to $z'$. Let $t_0,t_1> 0$ be such that $t_0+t_1= t$ and $u'-u_0 = t_0 (\min F) +t_1(\max F)$. The idea is the following: since we have the choice for $v$, we can essentially drive $x(t)$ to any position, so we put it first at a minimum of $F$ for a time $t_0$, then at a maximum of $F$ for a time $t_1$, and finally we bring it to the end point $x'$.
  	  
  	  More precisely, for any $\varepsilon \in [0,\min(t_0,t_1)/2)$, let $y_0,y_1,y_2\in\mathbb R $ be such that $F(x+y_0)=\min F$, $F(x+y_0+y_1) = \max F$ and $x+y_1+y_2+y_3 = x'$. Set $v(s) = y_0/\varepsilon$ for $t\in[0,\varepsilon]$, $v(s)=0$ for $t\in(\varepsilon,t_0]$, $v(s) = y_1/\varepsilon$ for $s\in(t_0,t_0+\varepsilon]$, $v(s) =  0$ for $s\in [t_0+\varepsilon,t_1-\varepsilon)$ and finally $v(s) = y_2/\varepsilon$ for $s\in[t_1-\varepsilon,t_1]$.  Let $s\mapsto z_\varepsilon(s):= (x_\varepsilon(s),u_\varepsilon(s))$ be the solution of the associated equation \eqref{EDOxu} with initial condition $z$. Remark that for all $s\in[0,t]$ and for all $\varepsilon>0$, $|u_\varepsilon(s)| \leq |u_0|+t\| F\|_\infty$, and in particular $u_\varepsilon(s) F'(x_\varepsilon(s))$ is bounded uniformly in $s$ and $\varepsilon$. As a consequence, $(x_\varepsilon(\varepsilon),u_\varepsilon(\varepsilon)) \rightarrow (x+y_0,u_0)$ as $\varepsilon \rightarrow 0$. Since $x+y_0$ is a minimum of $F$, $s\mapsto z_*(s):=(x+y_0,u_0 + s \min F)$ solves \eqref{EDOxu} with $v(s) = 0$, so that
  	  \[\underset{s\in[\varepsilon,t_0]} \sup | z_\varepsilon(s) - z_*(s)| \ \underset{\varepsilon\rightarrow 0}\longrightarrow \ 0\,.\]
  	  Then, similarly, $z_\varepsilon(t_0+\varepsilon) \rightarrow (x+y_0+y_1,u_0+t_0\min F)$ as $\varepsilon \rightarrow 0$ and thus
  	    	  \[\underset{s\in[t_0+\varepsilon,t_1-\varepsilon]} \sup | z_\varepsilon(s) - \tilde z_*(s)| \ \underset{\varepsilon\rightarrow 0}\longrightarrow \ 0\,.\]
where $s\mapsto \tilde z_*(s) := (x+y_0+y_1,u_0 + t_0 \min F + (s-t_0) \max F)$ solves again \eqref{EDOxu} with $v(s) = 0$. Finally $z_\varepsilon(t) \rightarrow z'$ as $\varepsilon$ vanishes, which concludes.
%  	  First, choose $v_0\in\mathbb{R}$ such that $v(s)=v_0$ for all $s\in [0,\varepsilon]$ and $x(\varepsilon)\in \{y\in\mathbb{S}^1\; s.t\; F(y)=\min F\}$ and let $v (s)=0$ for $s\in (\varepsilon, t_0]$. Then, pick $v_1\in\RR$ such that $v (s)=v_1$ for all $s\in (t_0,t_0+\varepsilon]$ and 
%  	  $x (t_0+\varepsilon)\in \{y\in\mathbb{S}^1\; s.t\; F(y)=\max F\}$ and let $v (s)=0$ for $s\in (t_0+\varepsilon, t-\varepsilon]$. Finally, choose $v_2\in\RR$ such that  $v (s)=v_2$ for all $s\in (t-\varepsilon,t]$ and $x(t)=x'$.
%  	  
%  	  Note that $u(t) = u' + \underset{\varepsilon\rightarrow0}o(1)$.   	  
  	 % From the definition of $u^v$, we have $\lim_{\varepsilon\rightarrow 0}u^{v_\varepsilon}(t)=u'.$
  	   
%  	    \begin{equation*}
 % 	    supp(P_t(z,dz'))\subset \mathbb{S}^1\times [u_0 +(\min F) t, u_0 +(\max F) t] .
  %	    \end{equation*}
  	%    This completes the proof.
  \end{proof}
  
\begin{proof}[Proof of Proposition \ref{PropNoyau} in the diffusion case]

Denoting by $p_t(\cdot,\cdot)$ the transition density given by Lemma \ref{densiteDiff}, let $z_1,z_2\in E$ be such that $p_{t_1}(z_1,z_2)>0$ for some $t_1>0$. By continuity, there exist neighbourhood $I_1$ and $I_2$ of respectively $z_1$ and $z_2$  such that the infimum  of $p_{t_1}$ over $I_1\times I_2$ is $c>0$.

Let $\mathcal K$ be a compact set and let $t_0$ be large enough so that
\begin{equation*}
I_1 \cap\po  \mathbb{S}^1\times\underset{(x,u)\in \mathcal K}{\bigcap} [u+(\min F)t_0, u+(\max F)t_0]\pf
\end{equation*}
has a non-empty interior. For $t\geq t_0$, the continuity of $p_{t}$ and the compactness of $\mathcal K$ imply
\begin{equation*}
c_t := \inf_{z\in \mathcal K} \mathbb{P}(Z_{t}\in I_1\ \vert\ Z_0 =z)>0.
\end{equation*}
Let $\nu$ be the uniform measure on $I_2$, namely $\nu(D) = \frac{\lambda\po D \cap I_2\pf}{\lambda (I_2)}$ for any Borel set $D$ of $E$. Then for all $z\in \mathcal K$ and for $t\geq t_0+t_1$,
\begin{eqnarray*}
\mathbb P\po Z_{t}\in D \ \vert\ Z_0 =z\pf &\geqslant & \mathbb P\po Z_{t}\in D \ \vert\ Z_{t-t_1}\in I_1 \pf \mathbb P\po Z_{t_1 }\in I_1 \ \vert\ Z_0 =z\pf\\
&\geqslant &  c_{t-t_1} c \lambda(I_2) \nu (D). 
\end{eqnarray*}

\end{proof}  
  
%%%%%%%%%%%%%%%%%%
\subsection{For the velocity jump process}
%%%%%%%%%%%%%%%%%%

In this subsection we consider the process $Z=(X,Y,U)$ with generator~\eqref{MODELPDMP}. The construction of a trajectory is similar to the one exposed in Section \ref{SectionHittingPDMP}, except from these slight modifications: in the definition of $\theta_1$, $g(T_k+s)$ is replaced by $U_{T_k} + \int_0^s F\po X_{T_k}+u Y_{T_k}\pf du$ and between the two jump times $T_k$ and $T_{k+1}$, $U$ is defined by $U_t= \int_{T_k}^t F(X_s)ds$.

We start with a controllability result.  
\begin{lem}\label{LemControlePDMP}
	Let $\mathcal K$ and $\mathcal V$ respectively be a compact and open set of $\mathbb S^1 \times \R \times \{-1,1\}$. Then there exists $t_0>0$ such that, for all $t\geq  t_0$,
	\begin{eqnarray*}
		\underset{z\in\mathcal K}\inf \mathbb P\po Z_{t} \in\mathcal V\ |\ Z_0 = z\pf & > & 0. 
	\end{eqnarray*}
\end{lem}

\begin{proof}
The boundedness of $F$ implies that for $t>0$, there exists a compact set $\mathcal K_2$ such that for all $s<t$ and for all $z_0\in \mathcal K$, if $Z_0=z_0$ then $Z_s\in\mathcal K_2$. Hence results from \cite{BLBMZ3} apply even if our whole state space is not compact. In particular, the process is Feller, and because $\mathcal K$ is compact we only need to prove that, for $t$ large enough, $$\mathbb P\po Z_{t} \in\mathcal V\ |\ Z_0 = z\pf>0$$ for all $z\in\mathcal K$. Let $z_0=(x_0,y_0,u_0)\in\mathcal K$ and $z_1=(x_1,y_1,u_1)\in\mathcal V$.  We proceed in three steps. % Let  $z_1\in V$ be in the interior of $V$ and let $z\in K$.

First, suppose that we can choose  in a deterministic way  a piecewise constant velocity $y(s)\in\{-1,0,1\}$ for $s\in[0,t]$, from which $\po x(s),u(s)\pf_{s\in[0,t]}$ is defined by an initial condition and by the ODE 
\begin{equation}\label{EqControlePDMP}
\begin{pmatrix}
\dot{x}\\
\dot{u}
\end{pmatrix}=\begin{pmatrix}
y\\
F(x)
\end{pmatrix}.
\end{equation}	
Let $h_0<h_1<h_2<h_3$ be such that $F(x_0+h_0) = \min F$, $F(x_0+h_1) = \max F$, $F(x_0+h_2)=0$ and $x_0+h_3 = x_1$.
For  $t_0>h_3$ large enough and any $t\geqslant t_0$, we can build a path of length $t$ between $z_0$ and $z_1$ as follows. 
%In this case it is not difficult to bring the process from any point $(x_0,u_0)$ to any point $(x_1,u_1)$ in a sufficiently large time $T$: 
 Given $0<s_1<s_2<t-h_3$, denote
\[\mathcal I = [0,h_0)\cup [h_0+s_1,h_1+s_1)\cup [h_1+s_2,h_2+s_2)\cup [t-(h_3-h_2),t)\]
set $y(s) = 1$ for $s\in\mathcal I$, $y(s) = 0 $ for $s \in [0,t)\setminus \mathcal I$ and $y(t) = y_1$ and let $s\mapsto (x(s),u(s))$ be the solution of the associated system \eqref{EqControlePDMP} with initial condition $(x_0,u_0)$. In particular,
\[x(t) = x_0 + h_0 + (h_1-h_0) + (h_2-h_1) + (h_3-h_1) = x_1\,,\]
and 
\[u(t) = u_0 + \int_0^{h_3} F(x+s)d s + s_1 \max F + (s_2-s_1) \min F \,. \]
For $t_0$ large enough, there exist $s_1<s_2 <t_0$ such that $u(t) = u_1$. This gives a path from $z_0$ to $z_1$ that solves \eqref{EqControlePDMP} with velocities $y(s) \in\{-1,0,1\}$.
%
%If $u_1>u_0$ (resp. $u_1<u_0$), choose a non-zero velocity to bring $x(s)$ to a point $x^*\in M(F,+)$ (resp. $x^\ast\in m(F,-)$). Then, pick the zero velocity and wait until $u(s)$ reaches the value $u_1 - \int_{x^*}^{x_1} F(s) \dd s$. Next, with the velocity $y(s)=1$, bring $x(s)$ to a point $x^{\ast\ast}$ such that $F(x^{**})=0$ and wait up to the time $t- |x^{**}-x_1|$. Finally, with the velocity $y(s)=1$, push $x(s)$ to $x_1$, and set the velocity to $y_1$ at time $t$.
	
	In a second instance,  we can   choose a deterministic $y(s)\in\{-1,1\}$ such that the solution of the system \eqref{EqControlePDMP} starting from $z_0$ is arbitrarily close to $z_1$ at time $t_0$. To ensure this, we simply approximate the case $y(s)=0$ in the previous step by sufficiently fast and balanced jumps between $-1$ and $1$.  
	
	Finally, we consider the PDMP %(namely if the jumps are random)
 starting from $z_0$. Since the random jump times have positive density, the PDMP follows arbitrarily closely a trajectory as described in the second step with positive probability. Hence, given any neighbourhood of $z_1$, the PDMP has positive probability to be in it at time $t_0$, which concludes.
 
 %Therefore, if $z_1\in\mathcal{V}$, there exist $t_0$, depending only on $\mathcal{K}$ and $\mathcal{V}$, such that for all $T\geqslant t_0$, $Z_{t_0}\in\mathcal{V}$ with positive probability from any initial condition in $\mathcal{K}$.
 % there is a positive probability to follow a trajectory which is arbitrarily close to the previous one, so that there is a positive probability to be in any given neighborhood of $z_1=(x_1,y_1,u_1)$  at time $T$, which conclude if $z_1\in\mathcal V$.
\end{proof}

\begin{proof}[Proof of Proposition \ref{PropNoyau} in the PDMP case]

	Consider the following vector fields:
	%\[G_{1} = \partial_x + F(x) \partial_u\hspace{20pt} G_{-1} = -\partial_x + F(x) \partial_u.\]
	\begin{equation*}
	G_{-1}(x,u)=  \begin{pmatrix}
	-1 \\
	F(x)
	\end{pmatrix}\text{ and } G_1(x,u)= \begin{pmatrix}
	1 \\
	F(x)
	\end{pmatrix}.
	\end{equation*}
	Then their difference is 
	\begin{equation*}
	G_{1}-G_{-1}=\begin{pmatrix}
	2 \\
	0
	\end{pmatrix}
	\end{equation*}
	so that the Lie bracket $[G_{1}-G_{-1},G_1](x,u)$ is 
	\begin{equation*}
	[G_{1}-G_{-1},G_1]= 2\partial_x G_1(x,u)= \begin{pmatrix}
	0 \\
	2 F'(x)
	\end{pmatrix}
	\end{equation*} 
	Since $F$ is not constant and smooth, there exists some $x$ such that $F'(x)\neq 0$, at which point the rank of $\po G_{1}-G_{-1},[G_{1}-G_{-1},G_1]\pf$ is 2.
	
	According to \cite[Theorem 4.4]{BLBMZ3}, it implies there exist a non-empty open set $\mathcal U$, a probability measure $\nu$  and $t_1,c>0$ such that $\forall z \in \mathcal{U}$  ,
	\begin{eqnarray*}
		\mathbb P\po Z_{t_1} \in \cdot \ |\ Z_0 = z\pf & \geqslant & c\ \nu(\cdot ). 
	\end{eqnarray*} 
	Considering $t_0>0$ given by Lemma \ref{LemControlePDMP} with $\mathcal V = \mathcal U$, we get that for any $z\in \mathcal K$, any Borel set $D$ and any $t\geqslant t_0$, %we write
	\begin{eqnarray*}
		\mathbb P\po Z_{t+t_1} \in D \ |\ Z_0 = z\pf &\geq & \mathbb P\po Z_{t} \in \mathcal U \ |\ Z_0 = z\pf \times \underset{z'\in\mathcal U}\inf \mathbb P\po Z_{t+t_1} \in D\ |\ Z_{t} = z'\pf \\
		& \geq & \po \underset{z'\in\mathcal K}\inf \mathbb P\po Z_{t} \in\mathcal U\ |\ Z_0 = z'\pf\pf c\ \nu(D)
	\end{eqnarray*}
	and Lemma \ref{LemControlePDMP} concludes.
\end{proof}
%%%%%%%%%%%%%%%%%
\section{Proof of the main theorem}\label{SectionPreuve}
%%%%%%%%%%%%%%%%
In this section we consider either $Z=(X,U)$ or $Z=(X,U,Y)$ such as in Theorem~\ref{ThmPrincipal}, and we call $E$ the state space, namely either $\mathbb S^1\times \RR_+$ or $\mathbb S^1\times \RR_+\times\{-1,1\}$.
%%%%%%%%%%%%%%%%%%%%%%
%\subsection{Ergodicity when  $\mathcal M=\emptyset$}
%%%%%%%%%%%%%%%%%%%%%%

The case $\mathcal M=\emptyset$ is a classical consequence of Harris' ergodic theorem:

\begin{proof}[Proof of point 1 of Theorem \ref{ThmPrincipal}]
Let $\kappa_0$ be given by Proposition~\ref{PropLyapunov} and $\kappa \leq \kappa_0$ be small enough so that $\mathbb E\po e^{\kappa |U_0|}\pf < \infty$. Let $t_0$ be large enough for both Propositions~\ref{PropLyapunov} and \ref{PropNoyau} to apply.  Let $\mathcal P_t$ be the Markov kernel on $E$ defined by $\mathcal P_t f(z) = \mathbb E \po f(Z_{t})\ | \ Z_0 = z\pf$. %Let $t_0$ be large enough for both Propositions~\ref{PropLyapunov} and \ref{PropNoyau} to apply for $t\geq t_0$, and let $V(z) = \exp(\kappa|z|)$. 
Then \cite[Theorem 1.2]{HaiMat}, applied to $\mathcal P_{t_0}$ with $V(z) = \exp(\kappa|z|)$, implies that $\mathcal P_{t_0}$ admits a unique invariant measure $\mu$ and that there exists constants $C>0$ and $\gamma\in(0,1)$ such that for all $n\in\mathbb N$,
\begin{eqnarray*}
d_{TV}\po Law(Z_{n t_0}),\mu\pf  & \leq & C \gamma^{n} \mathbb E \po e^{\kappa|U_0|}\pf.
\end{eqnarray*}

 By the semi-group property, for all $t\geq 0$, $\mu \mathcal P_t \mathcal P_{t_0} =  \mu \mathcal P_{t_0}\mathcal P_t = \mu \mathcal P_t$, so that $\mu \mathcal P_t$ is invariant for $\mathcal P_{t_0}$ and hence, by uniqueness,  $\mu\mathcal P_t = \mu$. In other words, $\mu$ is invariant for $\mathcal P_t$ for all $t\geq 0$, and in particular 
\begin{eqnarray*}
d_{TV}\po Law(Z_t),\mu\pf & = & d_{TV}\po Law(Z_t),\mu \mathcal P_{t-\lfloor t/t_0\rfloor t_0} \pf \\
& \leq &  d_{TV}\po Law(Z_{\lfloor t/t_0\rfloor t_0}),\mu\pf \\
  & \leq & C \gamma^{-1} e^{-t |\ln \gamma |/t_0 }\mathbb E \po e^{\kappa|U_0|}\pf.
\end{eqnarray*}

Finally, the controllability results of Section~\ref{SectionDensite} (Lemmas \ref{densiteDiff} and \ref{LemControlePDMP}) imply that $\mu$ has full support.
\end{proof}

%%%%%%%%%%%%%%%%%%%%%
%\subsection{Localization when $\mathcal M\neq \emptyset$}
%%%%%%%%%%%%%%%%%%%%%%
The rest of the section is devoted to the proof of the localization of the processes  when $\mathcal M\neq \emptyset$.

%By similarity of the arguments, we assume $m(F,+)$ to be non-empty.
\begin{prop}\label{PropLocal}
Suppose $m(F,+)\neq \emptyset$. Then there exist $p>0$ and $M>0$ (which does not depend on $Z_0$) such that  if $X_0=x_0 \in m(F,+)$ and $U_0\geq M$, then
\begin{eqnarray*}
\mathbb P\po X_t \underset{t\rightarrow \infty }\longrightarrow x_0\pf &\geqslant & p. 
\end{eqnarray*}
\end{prop}

\begin{proof}
For $j\geqslant 0$, define
\begin{equation*}
\eta_j = \frac{4\ln(1+j)}{1+j}\wedge \delta , 
\end{equation*}
set $c=\max\{ \frac{1}{F(x)},\ x\in m(F,+)\}$ and $S_0=0$ and define by induction the following stopping times:
\begin{eqnarray*}
\tau_{j+1}&=&\inf\left\{t>S_j \, :\,  X_t\in C^{\eta_{j+1}}\right\},\\
\tilde{S}_{0,j}&=& S_j,\\
\tilde{T}_{k,j}&=&\inf \left\{t>\tilde{S}_{k-1,j} \, : \,  X_t\in B^{\eta_{j+1}}\right\}\wedge (\tilde{S}_{k-1,j}+c)\wedge \tau_{j+1},\; k\geqslant 1,\\
\tilde{S}_{k,j}&=&\inf \left\{t>\tilde{T}_{k,j} \, :\,  X_t\in A\right\}\wedge \tau_{j+1},\; k\geqslant 1\,,
\end{eqnarray*}
and $S_{j+1}=\tilde{S}_{N_j,j}$ with 
\[N_j \ = \ \inf\left\{k\in\mathbb N \, : \,  \tilde{S}_{k,j}\geqslant S_j +c \text{ or } \tilde{S}_{k,j}=\tau_{j+1}\right\}\,.\]
  %Finally let 
%\begin{eqnarray*}
%	\alpha & =&\inf\left\{k\in\mathbb N,\; X_{t_k}\in C^{\eta_k}\right\}.
%\end{eqnarray*}

Let us give some intuition on these definitions. The connected component of $\po C^{\eta_j}\pf^c$ that contains $x_0$ is a neighbourhood of $x_0$ whose diameter goes to 0 as $j$ goes to $\infty$. At time $\tau_j$, the process has escaped from this neighbourhood. For $t\leq \tau_j$, the process makes possibly many oscillations near $x_0$.  When such an oscillation is large enough for the process to reach $B^{\eta_j}$ (this is at a time $\widetilde T_{k,j}$ for some $k$), we consider  this is the beginning of an attempt to leave $\po C^{\eta_j}\pf^c$. If this attempt fails, the process falls back to $x_0$ (this is $\widetilde S_{k,j}$). While $X$ makes those attempts to escape, time goes by, so that $U$ increases: after a time $c$, $U$ has increased at least by 1. Next time $X$ falls back to $x_0$ (this is $S_{j+1}$), we shrink the neighbourhood, namely from then we consider that the process escapes if it reaches $C^{\eta_{j+1}}$. From $S_j$ to $S_{j+1}$, there have been $N_j$ attempts to leave. The sequence $\eta$ is scaled so that there is in fact a positive probability that the process never escape from the shrinking neighbourhood that collapses at infinity to $\{x_0\}$.

Let us write these ideas more precisely. Note that as long as $S_{j+1}<\tau_{j+1}$, 
\begin{equation*}
S_{j+1}-S_j\geqslant c \hspace{21pt}\text{ and }\hspace{21pt} U_t\geqslant M+j 
\end{equation*}
for $t\geqslant S_{j}$. We take $M$ large enough so that $\po M+j\pf \eta_j>1$ for all $j\in\mathbb N$. Therefore, from Proposition \ref{PropHitting}, for all $k\geq 1$,
\begin{eqnarray*}
\mathbb{P}(\tilde{S}_{k,j}=\tau_{j+1}\vert \; \tilde{T}_{k,j}<\tau_{j+1})
&\leqslant& K(j+M)e^{-(j+M)\eta_{j}}.
\end{eqnarray*}
%for some deterministic constants $K_1$ and $K_2$ which do not depend on $j$. This 
It implies that $\Big(\mathbbm{1}_{\tilde{S}_{(i\wedge N_j) ,j}<\tau_{j+1}}+(i\wedge N_j)K(j+M)e^{-(j+M)\eta_{j}}\Big)_{i\geqslant 0}$ is a submartingale. Thus,
\begin{eqnarray}\label{relBLOCage}
\mathbb{P}(S_{j+1}<\tau_{j+1}\vert\;  S_j<\tau_j)&=& 1+ \mathbb{E}(\mathbbm{1}_{S_{j+1}<\tau_{j+1}}-\mathbbm{1}_{S_{j}<\tau_{j}}\vert\;  S_j<\tau_j)\notag\\
&\geqslant & 1- K(j+M)e^{-(j+M)\eta_{j+1}}\mathbb{E}(N_j\vert  \; S_j<\tau_j).
\end{eqnarray}
%with $K_3=K_1+K_2$.
From Proposition \ref{PropHitting}, we have %there exists $N\geqslant 1$ and a positive random variable $R$ satisfying a Large Deviation Principle such that 
\begin{equation*}
\tilde{S}_{k+1 ,j}-\tilde{S}_{k,j}\overset{sto}{\geqslant}  \eta_j  R.
\end{equation*}
Hence, considering a sequence  $(R_i)_{i\in\NN}$  of i.i.d random variables distributed like $R$
\begin{eqnarray*}
N_j&\overset{sto}{\leqslant}&\inf \left\{n\geqslant 1 \, :    \; \eta_j\sum_{i=1}^n R_i \geqslant c\right\}\\
&\leqslant & \left\lceil\frac{2c}{\mathbb{E}(R_1)\eta_j}\right\rceil + \inf \left\{n\geqslant 1 \, : \; \frac{1}{n}\sum_{i=1}^n R_i\geqslant \frac{\mathbb{E}(R_1)}{2}\right\}.
\end{eqnarray*}

Since $R$ is a positive r.v. with an exponential moment, from Cramer's Theorem (see e.g \cite[Chapter 2.4]{LDPbook} with the exercise 2.28 in it), there exist $c_1,c_2>0$ such that for all $n\geq 0$,
\[\mathbb P\po \frac{1}{n}\sum_{i=1}^n R_i\leqslant \frac{\mathbb{E}(R_1)}{2}\pf \leq c_1 e^{-c_2 n}.\] %satisfies a Large Deviation Principle, %so does $Z_{n}\overset{def}{=}\frac{1}{n}\sum_{i=1}^n \Big(\frac{R_i}{\mathbb{E}(R_i)}-1\Big)$.
Hence, applying the general formula $\mathbb E(J) = \sum_{k\in\mathbb N} \mathbb P(J\geqslant k)$ for a random variable $J$ on $\mathbb N$, we get
\[\mathbb{E}(N_j\vert  S_j<\tau_j)\ \leqslant \  \left\lceil\frac{2c}{\mathbb{E}(R_1)\eta_j}\right\rceil +\sum_{n\geqslant 1} \mathbb{P}\po \frac{1}{n}\sum_{i=1}^n R_i\leqslant \frac{\mathbb{E}(R_1)}{2}\pf 
\ \leqslant \  \frac{K'}{\eta_j}\]
for some constant $K'$ which does not depend on $j$, nor $M$. Thus  \eqref{relBLOCage} is now
\begin{eqnarray*} 
\mathbb{P}(S_{j+1}<\tau_{j+1}\vert\;  S_j<\tau_j)
&\geqslant & 1- \frac{K'K}{ \eta_j}(j+M)e^{-(j+M)\eta_{j+1}}.
\end{eqnarray*}
Take $M$ large enough so that the right-hand side is positive for all $j\in\mathbb N$. Then by induction 
\begin{eqnarray*}
	\mathbb{P}(S_{j+1}<\tau_{j+1})&=& \mathbb{P}(S_{j+1}<\tau_{j+1} \mid S_j<\tau_{j}) \ \mathbb{P}(S_j<\tau_{j})\notag\\
	&\geqslant&\prod_{i=0}^{j+1}\po  1- \frac{K'K}{\eta_i}(i+M)e^{-(i+M)\eta_{i+1}}\pf .
\end{eqnarray*}
As $(\{S_j<\tau_{j}\})_{j\geqslant 1}$ is a decreasing family of events,
\begin{eqnarray*}
\mathbb{P}(S_j<\tau_{j} \ \forall j\in\mathbb N)&=&\lim_{j\rightarrow \infty}\mathbb{P}(S_{j}<\tau_{j})\notag\\
&\geqslant&\prod_{j=0}^{\infty}\po  1- \frac{K'K}{\eta_j}(j+M)e^{-(j+M)\eta_{j+1}}\pf \\
%&=& \prod_{j\geqslant 1} (1-2ma_j\delta e^{-\frac{3a_j\delta^2}{4}})\notag\\
&=& \exp\po\sum_{j\geqslant 0} \ln\po  1- \frac{K'K}{\eta_j}(j+M)e^{-(j+M)\eta_{j+1}}\pf\pf.
%&\sim & \exp\Big(-mK_4K_3 \sum_{j\geqslant 1} j \eta_{j}^{-1/2N}e^{-j\eta_{j}}\Big)>0
\end{eqnarray*} 
For $j$ large enough,
\begin{eqnarray*}
 \frac{K'K}{\eta_j}(j+M)e^{-(j+M)\eta_{j+1}}  &\leqslant & \frac{1}{j^2}\,,
\end{eqnarray*} 
(where we used that $\eta_j = 4\ln(1+j)/(1+j)$  for $j$ large enough so that the right-hand-side is equivalent to $K'K/(j^2\ln j)$), and
\begin{eqnarray*}
\ln\po 1- \frac{K'K}{\eta_j}(j+M)e^{-(j+M)\eta_{j+1}} \pf &\geqslant & -\frac{1}{2 j^2}
\end{eqnarray*} 
which means $\mathbb{P}(S_j<\tau_{j} \ \forall j\in\mathbb N) >p>0$ where $p$ does not on depend $Z_0$. Yet,
\[\{S_j<\tau_{j} \ \forall j\in\mathbb N\} = \{ \forall j\in\mathbb N,\ \forall s\geq S_j,\ X_s\in I_{x_0}^{\eta_j}\}\]
and the $S_j$'s are all a.s. finite, which concludes.
\end{proof}
%\textbf{Remark:} 
\bigskip

\noindent\textbf{Remark :} The proof even  provides an estimation of the speed of convergence. Indeed we can see that $S_{j+1} \overset{sto}{\leqslant} S_j + c+ \delta R $, so that the $S_j$'s grow linearly to infinity.  From the non-degeneracy assumption on $F$, there exist $n\in\mathbb N$ and $c>0$ such that the diameter of $I^{\eta_j}_{x_0}$ is less than $c \eta_j^{\frac1n}$, depending on the first derivative of $F$ at $x_0$ to be non-zero (if $F$ is a Morse function, $n=2$). It means when there is convergence, it occurs at least at a speed of order $\po \frac{\ln t}{t}\pf^{\frac1n}$. 
\begin{proof}[Proof of point 2 of Theorem \ref{ThmPrincipal}:]
First note that by changing $U$ and $F$ to their opposites, Proposition \ref{PropLocal} also says that if $M(F,-)\neq \emptyset$ then there exist $p,M>0$ such that if $U_0<-M$ and $X_0\in  M(F,-)$ then $X_t$ converges to $x_0$ with probability at least $p$. %Let $M$ be large enough so that this, together with Proposition \ref{PropLocal}, hold.

	For $M>0$ large enough, $\varepsilon>0$ small enough and $x\in \mathcal M$, let
	\begin{eqnarray*}
	\mathcal{V}_x^\varepsilon &=& \left\{z'\in E\, :\ |x'-x|<\varepsilon \text{ and } u' \times sign\po F(x)\pf > M\right\}\\
	\mathcal{V}^\varepsilon &=& \underset{x\in\mathcal M}\bigcup \mathcal{V}^\varepsilon_x 	.
	\end{eqnarray*}
  When $\varepsilon$ is fixed, for $M$ large enough, if the process starts in $\mathcal V^\varepsilon_x$, from Inequality \eqref{EqHittingProba} (which is written for $x\in m(F,+)$ but by symmetry, again, also holds for $x\in M(F,-)$) it has a probability at least $\frac12$ to hit $\mathcal V^0_x$ before leaving $\mathcal V^{2\varepsilon}_x$. Then from Proposition \ref{PropLocal}, $X$ has a probability at least $p$ to converge to $x$. By the Markov property, it is then sufficient to prove that the hitting time of $\mathcal V^\varepsilon$ is almost surely finite in order to obtain that $X$ will almost surely converge to some point of $\mathcal M$.
	
	Denoting by $\tau_D$ the first hitting time of a set $D$ and $	\mathcal K = \{z\in E\, : \ |u|\leqslant M \}$,
	let us prove that for all $z_0\in E$
	\begin{equation*}
	\mathbb{P}(\tau_{\mathcal{V}^\varepsilon}\wedge \tau_{\mathcal K} <\infty\vert Z_0=z_0)=1.
	\end{equation*}
To do so, consider the case $u_0>M$ (as before, the case $u_0<-M$ is obtained by symmetry). Consider a smooth $2\pi$-periodic function  $\tilde F$ such that  $\tilde F(x) = F(x)$ for all $x\in  \{x'\in\mathbb S^1\, :\ |x'-x_0|>\varepsilon \ \forall x_0 \in m(F,+)\}$ and such that all the local minima of $\tilde F$ are negative  (i.e. $m(\tilde F,+) = \emptyset$). Let $(\tilde Z_t)_{t\geqslant 0}$ be the process constructed like $(  Z_t)_{t\geqslant 0}$ but with the function $\tilde F$ rather than $F$. In particular, in the diffusion, we use the same Brownian motion in the SDE \eqref{EqOrigin31} for both processes, and in the PDMP case we use the same sequence of i.i.d. exponential variables. That way, $Z_t = \tilde Z_t$ up to time $\tau_{\mathcal{V}^\varepsilon}\wedge \tau_{\mathcal K} $. By Lemma \ref{LemmAtteinte}, the hitting time of $\mathcal K$ by $\tilde Z$, which is greater than $\tau_{\mathcal{V}^\varepsilon}\wedge \tau_{\mathcal K} $, is almost surely finite.

%	 As long as the process is in the complementary of $\mathcal K \cup \mathcal V^\varepsilon$, 
%	 
%	  it behaves as if we had $\mathcal M=\emptyset$ (more precisely we can define a potential $\widetilde F$ which is equal to $F$ away from $\mathcal M$ and which have no positive minimum and no negative maximum, from which we can define an associated process $\widetilde Z$ with an initial condition $\widetilde Z_0 = Z_0$ such that $\widetilde Z_t = Z_t$ as long as $Z\notin \mathcal K \cup \mathcal V^\varepsilon$). Denoting by $\tau_D$ the first hitting time of a set $D$, Proposition \ref{PropRecurrent} then implies
%	 the conclusion of Proposition~\ref{PropRecurrent} (together with Remark \ref{laRem}) holds so that, denoting by $\tau_D$ the first hitting time of a set $D$, we have
%	\begin{equation*}
%	\mathbb{P}(\tau_{\mathcal{V}^\varepsilon}\wedge \tau_{\mathcal K} <\infty\vert Z_0=z)=1.
%	\end{equation*}
%	for all $z\in E$ (more precisely: suppose $U_0>M$, the case $U_0<-M$ being symmetric. Then  we can define a potential $\widetilde F$ which is equal to $F$ away from $m(F,+)$ and which have no positive minimum, from which we can define an associated process $\widetilde Z$ with the initial condition $\widetilde Z_0 = Z_0$ such that $\widetilde Z_t = Z_t$ as long as $Z\notin \mathcal K \cup \mathcal V^\varepsilon$. Then $\widetilde Z$, to which Proposition~\ref{PropRecurrent} applies, hits $\mathcal K$ in a finite time).
	
		On the other hand, by Lemmas \ref{LemControlePDMP} (for the PDMP) and  \ref{densiteDiff} (for the diffusion), there exists $t_0>0$ such that for all $x\in\mathcal M$,
	\begin{eqnarray*}
		\underset{z\in\mathcal K}\inf \mathbb P\po Z_{t_0} \in\mathcal V^\varepsilon_x\ |\ Z_0 = z\pf & > & 0. 
	\end{eqnarray*}
	It therefore follows that for any $z\in E$, 
	\begin{equation*}
	\mathbb{P}(\tau_{\mathcal{V}^0}<\infty\vert Z_0=z)=1
	\end{equation*}
	and moreover 
		\begin{equation*}
	\mathbb{P}(X_{\tau_{\mathcal{V}^0}}= x)>0
	\end{equation*}
	for all $x\in \mathcal M$.  Proposition \ref{PropLocal} concludes.
\end{proof}
%For $x\in\mathcal{M}$, the interval $I_x^{\eta_0}$ will be called a \textit{trap}.

\section*{Acknowledgement}

The first author thanks Olivier Raimond for a discussion that led to this paper. The second author acknowledge financial support from French ANR PIECE. Finally, both authors thank the anonymous referees for their comments, and acknowledge financial support from the Swiss National Science Foundation, Grant 200021\_163072. This paper was written when both authors were affiliated to the University of Neuchâtel.

%%%%%%%%%%%%%%%%%%%%%%%%%%%%%%%%%%%%%%%%%%%%%%%%%%%%%%%%%%%%%%%%%%%%%%%%%%%%%%%%%%%%%%%%%%%%%%%%%%%%%%%%%%%%%%%%%%%%


\begin{thebibliography}{99}
\bibitem{CEGMB} Bena{\"i}m M. and Gauthier C.E.: \textit{Self-repelling diffusions on a Riemannian manifold}, Probab. Theory Relat. Fields., Volume 169, Issue 1-2, Page 63-104 (2017)

\bibitem{BLR} {Bena{\"{\i}}m M., Ledoux M. and Raimond O.}, \textit{Self-interacting diffusions}, Probab. Theory Related Fields, volume 122, number 1: 1-41 (2002)
%\bibitem{BH} Bena{\"i}m, M., Hirsch, M.W.: \textit{Asymptotic Pseudo-Trajectories and Chain-recurrent flows, with Applications }, Journal of Dynamics and Differential Equations. Vol 8, No 1 pp 141-174. (1996) 

\bibitem{BLBMZ3} {Bena{\"i}m} M., {Le Borgne} S., {Malrieu} F. and {Zitt} P.-A.,
\emph{Qualitative properties of certain piecewise deterministic Markov processes}, Annales de l'Institut Henri Poincar\' e (Probabilit\'es et Statistiques), 51(3): 1040-1075 (2015)	

\bibitem{BR} {Bena{\"{\i}}m M and Raimond O.}, \textit{Self-interacting diffusions III. Symmetric interactions}, Annals of Probab., volume 33, number 5: 1716-1759 (2005)
 	
\bibitem{Bovier} Bovier, A., Eckhoff, M., Gayrard, V.
              and Klein, M.:  \emph{Metastability in reversible diffusion processes. {I}. {S}harp asymptotics for capacities and exit times},  {J. Eur. Math. Soc.}, Volume 6, 399--424 (2004) 	
 	
\bibitem{CT} Ciesielski Z. and Taylor S.J.: \textit{First passage times and sojourn times for Brownian motion in space and the exact Hausdorff measure of the sample path}, Trans. Amer. Math. Soc. 103: 434-450 (1962). 	
 	
\bibitem{cranston} Cranston M. and Le Jan Y.: \textit{Self attracting diffusions : Two case studies}, Math. Ann. 303(1): 87-93 (1995) 

\bibitem{FGM} Fontbona J., Guerin H. and Malrieu F.: \textit{Long time behavior of telegraph processes under convex potentials}, Stoch. Proc. \& Applic. , Vol. 126, no. 10, 3077-3101 (2016)
 
\bibitem{G} Gauthier C-E.: \textit{Self attracting diffusions on a sphere and application to a periodic case}, Electronic Communications in Probability, Volume 21, Paper 53, 12pp (2016)

\bibitem{Hai} Hairer, M.: \textit{On Malliavin's proof of H\"ormander's theorem}, Bull. Sci. Math 135 (6-7): 650-666 (2011)

\bibitem{HaiMat} Hairer, M. and Mattingly, J. C., \textit{Yet another look at {H}arris' ergodic theorem for {M}arkov chains}, {Progr. Probab.}, Vol. 63,  {109--117} (2011)
 
 \bibitem{HerrRoy} Herrmann S. and Roynette B.: \textit{Boundedness and convergence of some self-attracting diffusions}, Math. Ann. 325(1): 81-96 (2003)
  
 
% \bibitem{Ichihara} Ichihara, K., Kunita, H.: \textit{A Classification of the Second Order Degenerate Elliptic Operators and its Probabilistic Characterization}, Z. Wahrscheinlichkeitstheorie und Verw. Gebiete, 30(3):235-254, 1974.
 
 \bibitem{IW} Ikeda, N., Watanabe, S.: \textit{Stochastic Differential Equations and Diffusion Processes}, North Holland, New-York, 1981 
 
 \bibitem{K-S} Karatzas, I., Shreve, S.E.: \textit{Brownian motion and stochastic calculus}, Springer-Verlag, New-York, second edition, 1991
 
 \bibitem{Khasminskii}   Khasminskii, R.: \emph{Stochastic  stability  of  Differential  Equations}, Second Edition, Stochastic Modelling and Applied Probability, Springer, volume 66 (2012). 
 
 %\bibitem{Klie} Kliemann, W.: \textit{Recurrence and Invariant Measures for Degenerate Diffusions}, Ann. Probab, volume 15, Number 2, pp 690 - 707 (1987) 
 
\bibitem{MonmarchePDMP} Monmarch\'e, P. \emph{{Piecewise deterministic simulated annealing}}. ALEA, Lat. Am. J. Probab. Math. Stat. 13 (1), 357-398 (2016) 

 \bibitem{LDPbook} Rassoul-Agha, F., Sepp\"al\"ainen, T.: \textit{A Course on Large Deviations with an Introduction to Gibbs Measures}, American Mathematical Society, Graduate Studies in Mathematics, Volume 162 (2015) 

\bibitem{Royer} Royer G., \textit{A remark on simulated annealing of diffusion processes},  {SIAM J. Control Optim.}, Volume 27, number 6, 1403-1408 (1989)

\bibitem{SV} Stroock, D.W., Varadhan, S.R.S.: \textit{On the support of diffusion processes with applications to the strong maximum principle},
Proc. of Sixth Berkeley Symp. Math. Statist. Prob., Volum 3, 333-359, Univ. California Press, Berkeley, 1972.


 \end{thebibliography}
\end{document}